\documentclass[aos,preprint]{imsart}
\RequirePackage[OT1]{fontenc}
\RequirePackage{amsthm,amsmath}
\RequirePackage[colorlinks,citecolor=blue,urlcolor=blue]{hyperref}
\usepackage{morefloats}
\usepackage{placeins}
\usepackage{amssymb}
\RequirePackage[authoryear]{natbib}
\usepackage{makeidx}
\usepackage{graphicx}
\usepackage{float}
\newtheorem{theorem}{Theorem}

\newtheorem{lemma}{Lemma}
\newtheorem{remark}{Remark}

\newcommand{\cas}{\overset{a.s.}{\rightarrow}} %
\newcommand{\cw}{\overset{d}{\rightarrow}}
\newcommand{\bbet}{\boldsymbol{\beta}} %
\newcommand{\icol}[1]{
  \left(\begin{smallmatrix}#1\end{smallmatrix}\right)%
}
\setattribute{journal}{name}{}
\begin{document}

\begin{frontmatter}
\title{Robust and sparse estimators for linear regression models}

\runtitle{Robust and sparse estimators for linear regression}

\begin{aug}
\author{\fnms{Ezequiel} \snm{Smucler}\thanksref{t1}\ead[label=e1]{esmucler@ic.fcen.uba.ar}} and
\author{\fnms{V\'\i{}ctor J.} \snm{Yohai}\thanksref{t2}\ead[label=e2]{vyohai@dm.uba.ar}}
\affiliation{Universidad de Buenos Aires}

\thankstext{t1}{Research supported by a doctoral scholarship from CONICET.}
\thankstext{t2}{Research partially supported by Grants W276 from Universidad of Buenos Aires, PIP’s 112-2008-01-00216 and 112-2011-01- 00339 from CONICET and PICT 2011-0397 from ANPCYT, Argentina.}

\address{
Ezequiel Smucler\\
Instituto de Calculo \\
Universidad de Buenos Aires\\
Ciudad Universitaria, Pabellon 2\\
Buenos Aires 1428\\
Argentina\\
\printead{e1}\\
}

\address{
V\'\i{}ctor J. Yohai\\
Instituto de Calculo \\
Universidad de Buenos Aires\\
Ciudad Universitaria, Pabellon 2\\
Buenos Aires 1428\\
Argentina\\
\printead{e2}\\
}
\end{aug}

\begin{abstract}
Penalized regression estimators are a popular tool for the analysis of sparse and high-dimensional data sets. However, penalized regression estimators defined using an unbounded loss function can be very sensitive to the presence of outlying observations, especially high leverage outliers. Moreover, it can be particularly challenging to detect outliers in high-dimensional data sets. Thus, robust estimators for sparse and high-dimensional linear regression models are in need.
In this paper, we study the robust and asymptotic properties of MM-Bridge and adaptive MM-Bridge estimators: $\ell_q$-penalized MM-estimators of regression and MM-estimators with an adaptive $\ell_t$ penalty. For the case of a fixed number of covariates, we derive the asymptotic distribution of MM-Bridge estimators for all $q>0$. We prove that for $q<1$ MM-Bridge estimators can have the \textit{oracle property} defined in \citet{Fan}. We prove that for all $t\leq 1$ adaptive MM-Bridge estimators can have the \textit{oracle property}. 
The advantages of our proposed estimators are demonstrated through an extensive simulation study and the analysis of a real high-dimensional data set.
\end{abstract}




\end{frontmatter}
\section{Introduction}
\label{sec-int}

In this paper, we consider the problem of robust and sparse estimation for
linear regression models. In modern regression analysis, sparse and high-dimensional estimation scenarios where ratio of the number of predictor variables to the number of observations, say $p/n$, is high, but the number of actually relevant predictor variables to the number of observations, say $s/n$, is low, have become increasingly common in areas such as bioinformatics and chemometrics. In this type of regression scenarios, due to the high-dimensional nature of the data, it is difficult to discover outlying observations using simple criteria. Traditional robust regression estimators do not produce sparse models and can have a bad behaviour with regards to robustness and efficiency when $p/n$ is high, see \cite{DCML}. Moreover, they cannot be calculated for $p>n$. Thus, robust regression methods for high-dimensional data are in need.

Modern approaches to estimation in sparse and high-dimensional
linear regression models include penalized least squares estimators, e.g. the LS-Bridge estimator of \cite{Frank} and the LS-SCAD estimator of \cite{Fan}. LS-Bridge estimators are penalized least squares estimators in which the penalization function is proportional to the $\ell_q$ norm with $q>0$. They include as special cases the LS-Lasso of \cite{lasso} ($q=1$) and the LS-Ridge of \cite{Ridge} ($q=2$). The LS-SCAD estimator is a penalized least squares estimator in which the penalization function, the SCAD, is a non-concave function with several interesting theoretical properties.

The theoretical properties of penalized least squares estimators have been extensively studied in the past years. Of special note is the so called \textit{oracle property} defined in \cite{Fan}: An estimator is said to have the \textit{oracle property} if the estimated coefficients corresponding to zero coefficients of the true regression parameter are set to zero with probability tending to one, while at the same time the coefficient corresponding to non-zero coefficients of the true regression parameter are estimated with the same asymptotic efficiency as if we knew the correct model in advance.

\cite{Knight} derive the asymptotic distribution of LS-Bridge estimators in the classical regression scenario of fixed $p$ and prove that for $q<1$ these estimators can have the \textit{oracle property}. They also show that for $q=1$, the LS-Lasso sets the estimated coefficients corresponding to zero coefficients of the true regression to zero with positive probability. The LS-Lasso estimator is not variable selection consistent unless rather stringent conditions are imposed on the design matrix, and thus in general does not posses the \textit{oracle property}; see \cite{Adaptive} and \cite{Buhlmann} for details. Moreover, the LS-Lasso estimator has a bias problem: it can excessively shrink large coefficients. To remedy this issue, \cite{Adaptive} introduced the adaptive LS-Lasso, where adaptive weights are used for penalizing different coefficients of the $\ell_1$ norm of the coefficients, and showed that the adaptive Lasso can have the \textit{oracle property}. 

We note that whereas there exist extremely efficient algorithms to calculate the LS-Lasso, see \cite{Buhlmann}, LS-Bridge estimators with $q<1$ seem to be somewhat difficult to calculate. An algorithm to calculate LS-Bridge estimators with $q<1$ is described in \cite{Huang}. As \cite{Adaptive} points out, adaptive LS-Lasso estimators can be calculated using any of the algorithms available to calculate LS-Lasso estimators.


Penalized least squares estimators are not robust and may be highly inefficient under heavy tailed errors. In an attempt to remedy this issue, penalized M-estimators defined using a convex loss function have been proposed; see for example \cite{lad-lasso} and \cite{sinica}. Unfortunately, these estimators are not robust with respect to contaminations in the predictor variables.

\cite{Croux} proposed the Sparse-LTS estimator, a least trimmed squares estimator with a $\ell_1$ penalization. In a simulation study, \cite{Croux} show that the Sparse-LTS can be robust with respect to contamination in both the response and predictor variables. The Sparse-LTS estimator can be calculated for $p>n$. However, \cite{Croux} do not provide any asymptotic theory for their estimator. \cite{RLARS} propose a robust version of the LARS procedure, see \cite{LARS}, and in extensive simulations show that the RLARS procedure produces well behaved estimators under diverse contamination models. However, since the RLARS procedure is not based on the minimization of a clearly defined objective function, a theoretical analysis of its properties is difficult. \cite{esl-lasso} proposed a penalized regression estimator based on an exponential squared loss function. They prove that a \emph{local minimum} of the objective function used to define their estimator can have the \textit{oracle property}. On the other hand, their proposed estimators cannot be calculated in regression scenarios with $p>n$. \cite{Maronna} introduced S-Ridge and MM-Ridge estimators: $\ell_2$-penalized S- and MM-estimators of regression. In extensive simulation studies he shows that these estimators can be robust in a variety of contamination scenarios. However, $\ell_2$-penalized regression estimators do not produce sparse models. \cite{Maronna} does not provide any asymptotic theory for these estimators.

In this paper,  we study the robust and asymptotic properties of MM-Bridge and adaptive MM-Bridge estimators: $\ell_q$-penalized MM-estimators of regression and MM-estimators with an adaptive $\ell_t$ penalty. We obtain lower bounds on the breakdown points of MM-Bridge and adaptive MM-Bridge estimators. For the case of a fixed number of covariates, we prove the strong consistency of MM-Bridge and adaptive MM-Bridge estimators under general conditions. We derive the asymptotic distribution of MM-Bridge estimators for all $q$ and prove that for $q<1$ they can have the \textit{oracle property}. For the special case of $q=1$ we show that the coordinates of the MM-Bridge estimator corresponding to null coefficients of the true regression parameter will be set to zero with positive probability. See the comments following Theorem \ref{Dist-Asin Knight}. We show that adaptive MM-Bridge estimators can have the \textit{oracle property} for all $t\leq 1$. We propose an algorithm to calculate both MM-Bridge estimators with $q=1$, which we call MM-Lasso estimators, and adaptive MM-Bridge estimators with $t=1$, which we call adaptive MM-Lasso estimators. Our algorithm uses the S-Ridge estimator of \cite{Maronna} as an initial estimator and iteratively solves a weighted-Lasso type problem. Even though we derive our asymptotic results for fixed $p$, MM-Lasso and adaptive MM-Lasso estimators can be calculated for $p>n$. In extensive simulations, we study the performance with regards to stability in the presence of high-leverage outliers, and prediction accuracy and variable selection properties for uncontaminated samples of the MM-Lasso and adaptive MM-Lasso estimators. Finally, we apply our proposed estimators to a real high-dimensional data set.

The rest of this paper is organized as follows. In Section \ref{sec-MM y S} we review the definition and some of the most important properties of MM and S-estimators of regression. In Section \ref{sec-MM-Bridge} we define S-Bridge, MM-Bridge and adaptive MM-Bridge estimators, we study their robust and asymptotic theoretical properties and we describe an algorithm to compute MM-Lasso and adaptive MM-Lasso estimators. In Section \ref{sec-sim} we conduct an extensive simulation. In Section \ref{sec-real} we apply the aforementioned estimators to a real high-dimensional data set. Conclusions are provided in Section \ref{conclusions}. Finally, the proof of all our results are given in the Appendix.

\section{MM and S-estimators of regression}
\label{sec-MM y S}

We consider a linear regression model with random carriers: we observe $(\mathbf{x}_{i}^{\text{T}},y_{i})$ $i=1,...,n,$
i.i.d. $(p+1)$-dimensional vectors, where $y_{i}$ is the response variable and
$\mathbf{x}_{i}\in\mathbb{R}^{p}$ is a vector of random carriers, satisfying
\begin{equation}
y_{i}=\mathbf{x}_{i}^{\text{T}}\boldsymbol{\beta}_{0}\mathbf{+}u_{i}\text{ for
}i=1,...,n, \label{ML}%
\end{equation}
where $\alpha_0$ and $\boldsymbol{\beta}_{0}\in\mathbb{R}^{p}$ are to be estimated and $u_{i}$ is independent of $\mathbf{x}_{i}$. For $\bbet \in \mathbb{R}^{p}$ let $\mathbf{r}(\bbet)=(r_{1}(\bbet),...,r_{n}(\bbet))$, where $r_{i}(\bbet)=y_i-\mathbf{x}_{i}^{T}\bbet$. 
Some of the coefficients of $\bbet_0$ may be zero, and thus the corresponding carriers do not provide relevant information to predict $y$. We do not know in advance the set of indices corresponding to coefficients that are zero, and it may be of interest to estimate it. For simplicity, we will assume $\boldsymbol{\beta}_{0}=(\boldsymbol{\beta}_{0,I},\boldsymbol{\beta}_{0,II})$, where $\boldsymbol{\beta}_{0,I}\in\mathbb{R}^{s}$, $\boldsymbol{\beta}_{0,II}\in\mathbb{R}^{p-s}$, all the coordinates of $\boldsymbol{\beta}_{0,I}\in\mathbb{R}^{s}$ are non-zero and all the coordinates of $\boldsymbol{\beta}_{0,II}\in\mathbb{R}^{p-s}$ are zero.

Let $F_{0}$ be the distribution of the errors $u_{i}$, $G_{0}$ the
distribution of the carriers $\mathbf{x}_{i}$ and $H_{0}$ the distribution of
$\mathbf{(x}_{i}^{\text{T}}\mathbf{,}y_{i}\mathbf{)}$. Then $H_{0}$ satisfies%
\begin{equation}
H_{0}(\mathbf{x}\mathbf{,}y)=G_{0}(\mathbf{x)}F_{0}(y-\mathbf{x}^{\text{T}%
}\boldsymbol{\beta}_{0}). \label{H0}%
\end{equation}
Let $\mathbf{x}_{I}$ stand for the first $s$ coordinates of $\mathbf{x}$ and let $G_{0,I}$ be its distribution. For $\mathbf{b} \in \mathbb{R}^{p}$ and $q>0$ we note
\begin{equation}
\Vert \mathbf{b} \Vert_{q}=\left(\sum\limits_{j=1}^{p} \vert \mathbf{b}_{j}\vert^{q}\right)^{1/q}
\nonumber
\end{equation}
and $\Vert \mathbf{b} \Vert=\Vert \mathbf{b} \Vert_{2}$. Throughout this paper, $\rho$-function will refer to a bounded $\rho$-function, in the sense of \cite{Libro}. A popular choice of $\rho$-functions is Tukey's Bisquare family of functions given by 
\begin{equation}
\rho^{B}_{c}(u)=1-\left(1-\left(\frac{u}{c}\right)^{2}\right)^{3}I(\vert u \vert \leq c),
\label{rho-marina}
\end{equation}
where $c>0$ is a tuning constant. 

Given a sample $\mathbf{u}=(u_{1},...,u_{n})$ from some distribution $F$ and $0<b < 1$ the corresponding
M-estimate of scale $s_{n}(\mathbf{u})$ is defined by
\begin{equation}
s_{n}(\mathbf{u})=\inf\left\{  s>0:\frac{1}{n}\sum\limits_{i=1}^{n}%
\rho\left(  \frac{u}{s}\right)  \leq b\right\}  . \nonumber%
\end{equation}
It is easy to prove that $s_{n}(\mathbf{u})>0$ if and only if $\#\{i:u_{i}%
=0\}<(1-b)n,$ and in this case%
\begin{equation}
\frac{1}{n}\sum\limits_{i=1}^{n}\rho\left(  \frac{u_{i}}{s_{n}(\mathbf{u}%
)}\right)  =b. \label{sn}%
\end{equation}

The robustness of an estimator is measured by its stability when a
small fraction of the observations are arbitrarily replaced by outliers
that may not follow the assumed model. A robust estimator should not be
much affected by a small fraction of outliers. A quantitative measure of an estimator's robustness, introduced
by \cite{Donoho}, is the finite-sample replacement breakdown point. 
Loosely speaking, the finite-sample replacement breakdown point of an estimator is the minimum fraction of outliers that may take the estimator beyond any bound.
For a regression estimator, this measure is defined as follows. Given a sample
$\mathbf{z}_{i}=(\mathbf{x}_{i}^{\text{T}},y_{i})$, $i=1,...,n$, let
$\mathbf{Z}=\{\mathbf{z}_{1},...,\mathbf{z}_{n}\}$ and let $\hat
{\bbet}(\mathbf{Z})$ be a regression estimator. The finite-sample
replacement breakdown point of $\hat{\bbet}$ is then defined as
\begin{equation}
FBP(\hat{\bbet})=\frac{m^{\ast}}{n},
\nonumber
\end{equation}
where
\begin{equation}
m^{\ast}=\max \left\{m\geq0:\hat{\bbet}(\mathbf{Z}_{m})\text{ is
bounded for all }\mathbf{Z}_{m}\in\mathcal{Z}_{m}\right\},
\nonumber
\end{equation}
and $\mathcal{Z}_{m}$ is the set of all datasets with at least $n-m$ elements
in common with $\mathbf{Z}$.

A breakdown point equal to $\varepsilon^{*}$ only guarantees that for any given contamination fraction $\varepsilon \leq \varepsilon^{*}$, there exists a compact set such that the estimator in question remains in that compact set whenever a fraction of $\varepsilon$ observations is arbitrarily modified. However, this compact set may be very large. Thus, although a high breakdown point is always a desirable property, an estimator that has a high breakdown point can still be
largely affected by a small fraction of contaminated observations.

Given a sample $(\mathbf{x}_{i}^{\text{T}},y_{i})$, $i=1,...,n$ from the model
given in (\ref{H0}), \cite{S} define the S-estimator of
regression as
\begin{equation}
\hat{\boldsymbol{\beta}}_{S}=\arg\min_{\boldsymbol{\beta}\in\mathbb{R}^{p}%
}s_{n}(\mathbf{r}(\bbet)),
\label{S-estim}
\end{equation}
where $s_n$ is a M-estimate of scale. \cite{Fasano} derive the asymptotic distribution of
S-estimators of regression under very general conditions. S-estimators can always be tuned so as to attain the maximum possible
finite-sample replacement breakdown point for regression equivariant estimators. However, S-estimators cannot combine high breakdown point with high
efficiency at the normal distribution, see \cite{Hossjer}.

Let $(\mathbf{x}_{i}^{\text{T}},y_{i})$, $i=1,...,n,$ be a sample satisfying
\eqref{H0}, and $\rho_0$ and $\rho_1$ be two $\rho$-functions satisfying $\rho_1 \leq \rho_0$. Then \cite{MM 87} defines the MM-estimator of regression as
\begin{equation}
\hat{\bbet}_{MM}=\arg\min_{\bbet\in\mathbb{R}^{p}%
}\sum\limits_{i=1}^{n}\rho_{1}\left(  \frac{r_{i}(\bbet)}{s_{n}%
(\mathbf{r}(\hat{\bbet}_{1}))}\right),
\label{def MM}
\end{equation}
where $\hat{\bbet}_{1}$ is a consistent and high breakdown point estimate of $\bbet_{0}$ and $s_{n}(\mathbf{r}(\hat{\bbet}_{1}))$ is the M-estimate of scale of the residuals of $\hat{\bbet}_{1}$, calculated using $\rho_0$ and $b$.

\cite{MM 87}, proves that under general conditions, MM-estimators are
strongly consistent for $\bbet_{0}$, and furthermore
\begin{equation}
\sqrt{n}(\hat{\bbet}_{MM}-\bbet_{0})\cw \mathit{N}_{p} \left(\mathbf{0}, s ( \bbet_{0} ) ^{2}\frac{a(\psi_1,F_{0})}{b(\psi_1,F_{0})^{2}}\mathbf{V}%
_{\mathbf{x}}^{-1}\right), \label{distasin MM}%
\end{equation}
where $\mathbf{V}_{\mathbf{x}}=E_{G_{0}}(\mathbf{xx}^{\text{T}})$, $s(\bbet_{0})$
is defined by
\begin{equation}
E_{H_0}\rho_0\left(\frac{y-\mathbf{x}^T\bbet_0}{s(\bbet_0)}\right)=b,
\nonumber
\end{equation}
\begin{equation}
a(\psi,F_{0})=E_{F_{0}}\psi^{2}\left(  \frac{u}{s(\bbet_0)}\right)
\nonumber
\end{equation}
and
\begin{equation}
b(\psi,F_{0})=
E_{F_{0}}\psi^{\prime}\left(  \frac{u}{s(\bbet_0)}\right).
\nonumber
\end{equation}

Besides, he shows that $\rho_{1}$ can be chosen so that the resulting
MM-estimator has simultaneously the two following properties:

\begin{itemize}
\item Normal asymptotic efficiency as close to one as desired.

\item Breakdown point greater than or equal to that of the initial estimator.
\end{itemize}

\citet{Libro} recommend to take $\rho_0 = \rho^{B}(u/c_{0})$ and $\rho_1 = \rho^{B}(u/c_{1})$.
The tuning constant for $\rho_0$, $c_0$, should be chosen so that the
resulting M-estimate of scale be consistent for the error standard deviation in the case of normal errors.
The choice of $c_1$ should aim at striking a balance between robustness and efficiency.
\citet{Libro} recommend to choose $c_{1}$ so that the
MM-estimator has an asymptotic efficiency of 85\% at the normal distribution.
The reason for choosing an 85\% asymptotic efficiency at the normal
distribution, is that at this level of the efficiency the MM-estimator has the
same maximum asymptotic bias as the initial S-estimator of regression for the
case of normal errors and normal carriers.

\section{S-Bridge, MM-Bridge and adaptive MM-Bridge estimators of regression}
\label{sec-MM-Bridge}

Given a sample $(\mathbf{x}_{i}^{\text{T}},y_{i})$, $i=1,...,n$, $\gamma_n>0$, $r>0$, a $\rho$-function $\rho_0$ and $0<b< 1$, we define the $\ell_r$-penalized S-Bridge estimator of regression following \cite{Maronna} as
\begin{equation}
\hat{\boldsymbol{\beta}}_{PS}=\arg\min_{\boldsymbol{\beta}\in\mathbb{R}^{p}%
} n \: s_{n}^2(\mathbf{r}(\boldsymbol{\beta}))+\gamma_n\Vert\boldsymbol{\beta}\Vert_{r}^{r},
\label{p-S}
\end{equation}
where $s_{n}(\mathbf{r}(\boldsymbol{\beta}))$ is the residual scale estimate defined using $\rho_0$ and $b$. If the model contains an intercept, then it is not penalized.

It is easy to see that
\begin{equation}
\Vert\hat{\bbet}_{PS}\Vert_{r}^{r} \leq \Vert\hat{\bbet}_{S}\Vert_{r}^{r},
\label{norm-ineq-S}
\end{equation}
where $\hat{\bbet}_{S}$ is the S-estimator calculated using $\rho_0$ and $b$.

Given another $\rho$-function $\rho_1$ which satisfies $\rho_1 \leq \rho_0$ and $\lambda_n >0$ we define the $\ell_q$-penalized MM-Bridge estimator of regression as
\begin{equation}
\hat{\boldsymbol{\beta}}_{B}=\arg\min_{\bbet\in\mathbb{R}^{p}%
}\sum\limits_{i=1}^{n}\rho_{1}\left(  \frac{r_{i}(\bbet)}{s_{n}%
(\mathbf{r}(\hat{\bbet}_{1}))}\right)+\lambda_n\Vert\boldsymbol{\beta}\Vert_{q}^{q},
\label{p-MM}
\end{equation}
where $\hat{\bbet}_{1}$ is a consistent initial estimate of $\bbet_0$. Clearly, the robustness of the MM-Bridge estimator will depend heavily on the robustness of the initial estimate. If the model contains an intercept, then it is not penalized. For the case $q=1$ we will call the resulting estimator MM-Lasso.

Note that our definition of a MM-Bridge estimator with $r=2$ and $q=2$, is not exactly the same as the definition of MM-Ridge estimators of \cite{Maronna}. For a given $\lambda_n$, the MM-Ridge of \cite{Maronna} is equal to our MM-Ridge estimator calculated with $\lambda_n/s_{n}%
(\mathbf{r}(\hat{\bbet}_{1}))^2$. Nonetheless, our asymptotic results can be very easily adapted to cover the MM-Ridge estimators as defined by \cite{Maronna}. However, this is not the case for our results concerning the finite-sample breakdown point of MM-Bridge estimators. \cite{Maronna} points out that the finite-sample breakdown point of MM-Ridge estimators is, for a fixed penalization parameter and according to his definition of MM-Ridge estimators, greater than or equal to the breakdown point of the residual scale $s_n$. In Theorem \ref{Theo-break-br}, we show that for a fixed penalization parameter and according to our definition of MM-Bridge estimators, the breakdown point of any MM-Bridge estimator is greater than $1-1/n$.
%

Given $\varsigma>0$, $t>0$ and $\iota_n$  we define the adaptive MM-Bridge estimator of regression as
\begin{equation}
\hat{\boldsymbol{\beta}}_{A}=\arg\min_{\bbet\in\mathbb{R}^{p}%
}\sum\limits_{i=1}^{n}\rho_{1}\left(  \frac{r_{i}(\bbet)}{s_{n}%
(\mathbf{r}(\hat{\bbet}_{1}))}\right)+\iota_n \sum\limits_{i=1}^{p} \frac{\vert \beta_j \vert^{t}}{\vert \hat{\beta}_{2,j} \vert^{\varsigma}},
\label{p-AMM}
\end{equation}
where $\hat{\bbet}_{2}$ is a consistent initial estimate of $\bbet_0$. Clearly if $\hat{\beta}_{2,j}=0$ for some $j$, then $\hat{\boldsymbol{\beta}}_{A,j}=0$. If the model contains an intercept, then it is not penalized. For the case $t=1$ we will call the resulting estimator adaptive MM-Lasso.
Note that for coefficients corresponding to large coefficients of $\hat{\bbet}_{2}$, the adaptive MM-Lasso employs a small penalty; this ameliorates the bias issues associated with the $\ell_1$ penalty. 

 \citet{esl-lasso} prove that their estimator can have the highest possible breakdown point among regression equivariant estimators, but it must be noted that their estimator is not regression equivariant. \cite{Croux} show that the breakdown point of the Sparse-LTS estimator is $(n-h+1)/n$, where $n-h$ is the number of trimmed observations, and prove that the breakdown point of the LS-Lasso estimator is $1/n$. Note that it follows immediately from \eqref{norm-ineq-S} that for any $\gamma_n$, the finite-sample breakdown point of $\hat{\bbet}_{PS}$ is at least as high as that of $\hat{\bbet}_{S}$. In Theorem \ref{Theo-break-br}, we prove that for any fixed $\lambda_n>0$, the breakdown point of $\hat{\bbet}_{B}$ is equal to $1-1/n$. In Theorem \ref{Theo-break-ad}, we prove that for any fixed $\iota_n>0$, the breakdown point of $\hat{\bbet}_{A}$ is greater than or equal to the breakdown point of $\hat{\bbet}_{2}$. However, one could argue that since $\hat{\bbet}_{PS}$, $\hat{\bbet}_{B}$ and $\hat{\bbet}_{A}$ are not regression equivariant, these results are rather vacuous. See \cite{davies-gather}.

\begin{theorem}
\label{Theo-break-br}
If $\lambda_n>0$ is fixed, then $FBP(\hat{\bbet}_{B}) \geq 1-1/n$
\end{theorem}

\begin{theorem}
\label{Theo-break-ad}
If $\iota_n>0$ is fixed, then $FBP(\hat{\bbet}_{A}) \geq FBP(\hat{\bbet}_{2})$.
\end{theorem}

Note that if $\hat{\bbet}_{2} = \hat{\bbet}_{B}$, then $FBP(\hat{\bbet}_{A})\geq 1-1/n$ whenever $\lambda_n,\iota_n >0$.
In practice, $\gamma_n$, $\lambda_n$ and $\iota_n$ may be chosen via some data-driven procedure such as cross-validation. In this case, the breakdown point of the resulting MM-Bridge and adaptive MM-Bridge estimators may be lower than $1-1/n$. The robustness of the resulting estimators will depend sorely on the robustness of the cross-validation scheme, and hence the use of robust residual scales as objective functions, instead of the classical root mean squared error, is crucial.

\subsection{Asymptotics}
\label{sec-asym}
We now describe the set-up to study the asymptotic properties of S-Bridge, MM-Bridge and adaptive MM-Bridge estimators of regression.
We will assume that
\begin{itemize}

\item[B1.] $\rho_0$ and $\rho_1$ are twice continuously differentiable and eventually constant.

\item[B2.] $P_{G_{0}}\left( \mathbf{x}^{T} \bbet=0\right)
<1-b$ for all non-zero $\bbet\in\mathbb{R}^{p}$.

\item[B3.] $F_{0}$ has an even continuous density, $f_{0}$, that is a monotone
decreasing function of $|u|$ and a strictly decreasing function of $|u|$ in a
neighborhood of 0.

\end{itemize}
A family of $\rho$-functions that satisfies [B1] is Tukey's Bisquare family of functions, given in \eqref{rho-marina}. Condition [B2] is needed in the proof of the consistency of S-Bridge estimators. Note that condition [B3] does not require finite moments from $F_0$. Thus, extremely heavy tailed error distributions, such as Cauchy's distribution, can be easily seen to satisfy [B3]. However, [B3] does impose a rather stringent symmetry assumption on the error distribution. This requirement greatly simplifies the asymptotic treatment of the estimators and is usual in robust statistics. 

The following theorem proves the strong consistency of S-Bridge, MM-Bridge and adaptive MM-Bridge estimators of regression whenever $\gamma_n=o(n)$, $\lambda_n=o(n)$ and $\iota_n=o(n)$ respectively.

\begin{theorem}
\label{Cons}
Let $(\mathbf{x}_{i}^{\text{T}},y_{i})$, $i=1,...,n,$ be i.i.d observations
with distribution $H_{0}$, which satisfies \eqref{H0}.  Assume [B1]-[B3] hold. Then 
\begin{enumerate}
\item[(i)] If $\gamma_n=o(n)$, $\hat{\boldsymbol{\beta}}_{PS} \cas \bbet_{0}$.
\label{S-cs}
\item[(ii)] If $\lambda_n=o(n)$, $\hat{\boldsymbol{\beta}}_{B} \cas \bbet_{0}$.
\label{MM-cs}
\item[(iii)] If $\iota_n=o(n)$, $\hat{\boldsymbol{\beta}}_{A} \cas \bbet_{0}$.
\label{AMM-cs}
\end{enumerate}
\end{theorem}

In practice, we will use the S-Ridge estimator of \cite{Maronna} as the initial estimate $\hat{\bbet}_{1}$ in \eqref{p-MM} and \eqref{p-AMM}. Note that according to Theorem \ref{Cons} and the remarks above Theorem \ref{Theo-break-br}, the S-Ridge is a high breakdown point and consistent estimate of $\bbet_0$, as long as the penalization parameter satisfies $\gamma_n=O(\sqrt{n})$.

In order to obtain the rate of convergence of MM-Bridge and adaptive MM-Bridge estimators we will have to make the following additional assumption:

\begin{itemize}
\item[B4.] $G_{0}$ has finite second moments and $\mathbf{V}_{\mathbf{x}}=E_{G_{0}}\mathbf{xx}^{\text{T}}$ is non-singular.
\end{itemize}

In the next theorem, we prove the $\sqrt{n}$-consistency of MM-Bridge and adaptive MM-Bridge estimators. 

\begin{theorem}
\label{Rate}
Let $(\mathbf{x}_{i}^{\text{T}},y_{i})$, $i=1,...,n,$ be i.i.d observations with distribution $H_{0}$, which satisfies \eqref{H0}.  Assume [B1]-[B4] hold. Then
\begin{enumerate}
\item[(i)] If $\lambda_n = O( \sqrt{n})$, then $\Vert \hat{\bbet}_{B} - \bbet_{0} \Vert =  O_{P}(1/\sqrt{n})$.
\item[(ii)] If $\iota_n = O( \sqrt{n})$, then $\Vert \hat{\bbet}_{A} - \bbet_{0} \Vert =  O_{P}(1/\sqrt{n})$.
\end{enumerate}
\end{theorem}

\begin{remark}
From now on, we will assume that the initial estimator used to define the penalty weights for the adaptive MM-Bridge estimator, $\hat{\bbet}_{2}$, is $\sqrt{n}$-consistent. For example, according to Theorem \ref{Rate}, we could take $\hat{\bbet}_{2}$ to be some MM-Bridge estimator calculated with $\lambda_n = O(\sqrt{n})$.
\end{remark}

Let $\hat{\bbet}_{A,I}$ stand for the first $s$ coordinates of $\hat{\bbet}_{A}$ and $\hat{\bbet}_{A,II}$ for the remaining $p-s$. 
Let $\hat{\bbet}_{B,I}$ stand for the first $s$ coordinates of $\hat{\bbet}_{B}$ and $\hat{\bbet}_{B,II}$ for the remaining $p-s$. The following theorem shows that, as long as $\varsigma > t-1$ and $t\leq1$, adaptive MM-Bridge estimators can be variable selection consistent, and that if $q<1$, then MM-Bridge estimators can be variable selection consistent as well. In particular, taking $t=1$, we prove the variable selection consistency of adaptive MM-Lasso estimators.

\begin{theorem}
\label{Spars}
Let $(\mathbf{x}_{i}^{\text{T}},y_{i})$, $i=1,...,n,$ be i.i.d observations with distribution $H_{0}$, which satisfies \eqref{H0}. Assume [B1]-[B4] hold.
\begin{enumerate}
\item[(i)] Suppose $q<1$, $\lambda_n=O(\sqrt{n})$ and $\lambda_n / n^{q/2} \rightarrow \infty$. 
Then \begin{equation}
\mathbb{P}\left( \hat{\bbet}_{B,II}=\mathbf{0}_{p-s}  \right) \rightarrow 1.
\nonumber
\end{equation}.
\item[(ii)] Suppose $t \leq1$, $\iota_n=O(\sqrt{n})$ and $\iota_n n^{(\varsigma - t)/2} \rightarrow \infty$. Then \begin{equation}
\mathbb{P}\left( \hat{\bbet}_{A,II}=\mathbf{0}_{p-s}  \right) \rightarrow 1. \nonumber
\end{equation}
\end{enumerate}
\end{theorem}

Next we derive the asymptotic distribution of $\hat{\bbet}_{B,I}$ and $\hat{\bbet}_{A,I}$.

\begin{theorem}
\label{Orac}
Let $(\mathbf{x}_{i}^{\text{T}},y_{i})$, $i=1,...,n,$ be i.i.d observations with distribution $H_{0}$, which satisfies \eqref{H0}. Assume [B1]-[B4] hold.
\begin{enumerate}
\item[(i)] Suppose $q<1$, $\lambda_{n}/\sqrt{n}\rightarrow 0$ and $\lambda_{n}/n^{q/2}\rightarrow \infty$. Then
\begin{equation}
\sqrt{n}(\hat{\bbet}_{B,I}-\boldsymbol{\beta}_{0,I}) \cw \mathit{N}_{s}\left(\mathbf{0},s(\bbet_0)^2\frac{a(\psi,F_{0})}{b(\psi,F_{0})^{2}}V_{\mathbf{x}_I}^{-1}\right).
\nonumber
\end{equation}
\item[(ii)] Suppose $t\leq1$, $\iota_{n}/\sqrt{n}\rightarrow 0$ and $\iota_{n} n^{(\varsigma-t)/2}\rightarrow \infty$.  Then
\begin{equation}
\sqrt{n}(\hat{\bbet}_{A,I}-\boldsymbol{\beta}_{0,I}) \cw \mathit{N}_{s}\left(\mathbf{0},s(\bbet_0)^2\frac{a(\psi,F_{0})}{b(\psi,F_{0})^{2}}V_{\mathbf{x}_I}^{-1}\right).
\nonumber
\end{equation}
\end{enumerate}
Here $a(\psi,F_{0})$ and $b(\psi,F_{0})$ are as in \eqref{distasin MM} and $V_{\mathbf{x}_I}=E_{G_{0}} \mathbf{x}_{I}\mathbf{x}_{I}^{T}$.
\end{theorem}

Theorem \ref{Spars} together with Theorem \ref{Orac} prove that $\hat{\bbet}_{A}$ and $\hat{\bbet}_{B}$ can have the \textit{oracle property} as long as $\varsigma>t-1$ and $t\leq1$, and $q<1$ respectively. That is: the estimated coefficients corresponding to null coordinates of the true regression parameter are set to zero with probability tending to 1, while at the same time the coefficients corresponding to non-null coordinates of the true regression parameter are estimated with the same asymptotic efficiency as if we had applied a non penalized MM-estimators to the relevant carriers only.

In Theorem \ref{Dist-Asin Knight} we derive the asymptotic distribution of 
$\hat{\bbet}_{B}$ for $q\geq1$. Our theorem is analogous to Theorem 2 of \cite{Knight}.

\begin{theorem}
\label{Dist-Asin Knight}
Let $(\mathbf{x}_{i}^{\text{T}},y_{i})$, $i=1,...,n,$ be i.i.d observations
with distribution $H_{0}$, which satisfies \eqref{H0}.  Let $q\geq 1 $. Assume [B1]-[B4] hold and $\lambda_n/\sqrt{n}\rightarrow \lambda_0$. Then
\begin{equation}
\sqrt{n}(\hat{\bbet}_{B}-\boldsymbol{\beta}_{0}) \cw \arg\min(R),
\nonumber
\end{equation}
where
\begin{equation}
R(\mathbf{z})=-\mathbf{z}^{T}\mathbf{W}+\frac{1}{2s(\bbet_{0})^{2}}b(\psi_0,F_{0})\mathbf{z}^{T}V_{\mathbf{x}}\mathbf{z}+\lambda_{0} q \sum\limits_{i=1}^{p} z_{j} sgn(\beta_{0,j}) \vert \beta_{0,j} \vert^{q-1},
\nonumber
\end{equation}
for $q>1$,
\begin{align*}
R(\mathbf{z})=-\mathbf{z}^{T}\mathbf{W}+\frac{1}{2s(\bbet_{0})^{2}}b(\psi_0,F_{0})\mathbf{z}^{T}V_{\mathbf{x}}\mathbf{z}+\lambda_{0}  \sum\limits_{i=1}^{p} ( z_{j} sgn(\beta_{0,j}) I(\beta_{0,j}\neq0)\\+ \vert z_{j} \vert I(\beta_{0,j} = 0) ),
\end{align*}
for $q=1$ and $\mathbf{W} \sim \mathit{N}_{p}\left(\mathbf{0},a(\psi_0,F_{0}) / s(\boldsymbol{\beta}_{0})^2 V_{\mathbf{x}}\right)$.
\end{theorem}

Note that if $\lambda_0=0$, $\hat{\boldsymbol{\beta}}_{B}$ has the same asymptotic distribution as the corresponding non-penalized MM-estimator. If $\lambda_0>0$ and $q=1$, the coordinates of $\hat{\bbet}_{B}$ corresponding to null coefficients of $\bbet_{0}$ will be set to zero with positive probability, the proof is essentially the same as the one that appears in pages 1361-1362 of \cite{Knight}. However, one can show that
\begin{equation}
\limsup \mathbb{P}\left(\hat{\boldsymbol{\beta}}_{B,II} = \mathbf{0}_{p-s} \right) \leq c <1,
\nonumber
\end{equation}
where $c$ depends on $\mathbf{V}_{x}$, $\lambda_0$ and $\bbet_{0}$. The proof is essentially the same as the proof of Proposition 1 of \cite{Adaptive}.

If $q>1$ the amount of shrinkage of the estimated regression coefficients increases with the magnitude of the true regression coefficients. Hence, for "large" parameters, the bias introduced by MM-Bridge estimators with $q>1$ may be unacceptably large, at least for the fixed $p$ scenario.  For the case $q=2$ we can calculate the asymptotic distribution of the estimator explicitly. It follows easily from Theorem \ref{Dist-Asin Knight} that the asymptotic distribution of the MM-Ridge estimator is
\begin{equation}
\mathit{N}_{p}\left(-2\lambda_0\frac{s(\boldsymbol{\beta}_{0})^2}{b(\psi_1,F_{0})}V_{\mathbf{x}}^{-1} \bbet_0, s(\boldsymbol{\beta}_{0})^2 \frac{a(\psi_1,F_{0})}{b(\psi_1,F_{0})^2} V_{\mathbf{x}}^{-1}\right).
\nonumber
\end{equation}

In the next theorem we derive the asymptotic distribution of 
$\hat{\bbet}_{B}$ for $q<1$ when $\lambda_n/n^{q/2}\rightarrow \lambda_0$. 

\begin{theorem}
\label{Dist-Asin Knight2}
Let $(\mathbf{x}_{i}^{\text{T}},y_{i})$, $i=1,...,n,$ be i.i.d observations
with distribution $H_{0}$, which satisfies \eqref{H0}.  Let $q< 1 $. Assume [B1]-[B4] hold and $\lambda_n/n^{q/2}\rightarrow \lambda_0$. Then
\begin{equation}
\sqrt{n}(\hat{\bbet}_{B}-\boldsymbol{\beta}_{0}) \cw \arg\min(R),
\nonumber
\end{equation}
where
\begin{equation}
R(\mathbf{z})=-\mathbf{z}^{T}\mathbf{W}+\frac{1}{2s(\bbet_{0})^{2}}b(\psi_0,F_{0})\mathbf{z}^{T}V_{\mathbf{x}}\mathbf{z}+\lambda_{0}  \sum\limits_{i=1}^{p}  \vert z_{j} \vert^{q}I(\beta_{0,j}=0),
\nonumber
\end{equation}
and $\mathbf{W} \sim \mathit{N}_{p}\left(\mathbf{0},a(\psi_0,F_{0})/ s(\boldsymbol{\beta}_{0})^2 V_{\mathbf{x}}\right)$.
\end{theorem}

It follows from Theorem \ref{Dist-Asin Knight2} that for $q<1$, if $\lambda_0>0$, the coordinates of $\hat{\bbet}_{B}$ corresponding to null coefficients of $\bbet_{0}$ will be set to zero with positive probability. Moreover, in this case the shrinkage only affects the coordinates of the estimators corresponding to null coefficients of $\bbet_0$, and hence no asymptotic bias is introduced.

\subsection{Computation}
\label{sec-comp}

In this section, we describe an algorithm to obtain approximate solutions of \eqref{p-MM} for $q=1$, i.e. MM-Lasso estimators. Through out this section we will assume that our model, \eqref{ML}, contains an intercept, and that the first coordinate of each $\mathbf{x}_i$ equals 1. Let $\mathbf{X}$ be the matrix with $\mathbf{x}_i$ as rows. 

Prior to any calculations, all the columns of $\mathbf{X}$, except the first one, are centered and scaled using the median and the normalized median absolute deviation respectively. The response vector $\mathbf{y}$ is centered using the median. At the end, the final estimates are expressed in the original coordinates. 

We take the S-Ridge estimator of \cite{Maronna}, which we note $\hat{\bbet}_{PS}$, as the initial estimate in \eqref{p-MM}. The penalization parameter for S-Ridge estimator, $\gamma_n$, is chosen via robust 5-fold cross-validation, as described in \cite{Maronna}. Let $s_n=s_{n}(\mathbf{r}(\hat{\bbet}_{PS}))$. 

Let $w(u)=\psi_1(u)/u$, where $\psi_1$ is the derivative of $\rho_1$. For a given $\bbet$, let $\omega_i=w(r_i(\bbet)/s_n)$. Suppose $\lambda_n$ is given.  and let $\mathbf{W}$ be the diagonal matrix formed by $\sqrt{\omega_1},..., \sqrt{\omega_n}$. Let $\mathbf{y}^{*}=\mathbf{W}\mathbf{y}$ and $\mathbf{X}^{*}=\mathbf{W}\mathbf{X}$. Let $\hat{\bbet}_{B}$ be the MM-Lasso estimator. It is easy to show that $\hat{\bbet}_{B}$ satisfies
\begin{equation}
\mathbf{X}^{*T}(\mathbf{y}^{*}-\mathbf{X}^{*}\bbet)+\lambda_n s_{n}^{2} \icol{0\\ sign(\beta_2)\\\vdots\\\ sign(\beta_{p+1})} \overset{s}{=} \mathbf{0}_{p+1},
\nonumber
\end{equation}
where $\overset{s}{=} \mathbf{0}_{p+1}$ stands for a change of sign.
Note that the first column of $\mathbf{X}^{*}$ equals $\mathbf{k}^{*}=(\sqrt{\omega_1},..., \sqrt{\omega_n})$. For each $j=2,...,p+1$ let $\mathbf{x}^{*(j)}$ be the $j$-th column of $\mathbf{X}^{*}$ and let
\begin{equation}
\eta_j=\frac{\mathbf{k}^{*T}\mathbf{x}^{*(j)}}{\Vert\mathbf{k}^{*}\Vert^{2} }.
\nonumber
\end{equation}
Then $\mathbf{x}^{*(j)}$ can be decomposed as the sum of two vectors: $\eta_j \mathbf{k}^{*}$, in the direction of $\mathbf{k}^{*}$, and $\mathbf{x}^{*\perp (j)}=\mathbf{x}^{*(j)}-\eta_j \mathbf{k}^{*}$, orthogonal to $\mathbf{k}^{*}$. Let $\mathbf{X}^{*\perp}$ be the matrix with columns $\mathbf{x}^{*\perp (2)},...,\mathbf{x}^{*\perp (p+1)}$. It is easy to show that $\hat{\bbet}_{B}$ satisfies
\begin{align}
\mathbf{k}^{*}\mathbf{y}^{*}-\Vert\mathbf{k}^{*}\Vert^{2}(\beta_1+\eta_2\beta_2+...+\eta_{p+1} \beta_{p+1})=0\label{wlasso-eq1} \\
\mathbf{X}^{*\perp T}(\mathbf{y}^{*}-\mathbf{X}^{*\perp T}\bbet)+\lambda_n s_{n}^{2}\icol{0\\ sign(\beta_2)\\\vdots\\\ sign(\beta_{p+1})} \overset{s}{=}  \mathbf{0}_{p+1}.\label{wlasso-eq2}
\end{align}

We note that if $\mathbf{k}^{*}, \mathbf{y}^{*}$ and $\mathbf{X}^{*\perp T}$ where known, $\hat{\bbet}_{B,2}$,..., $\hat{\bbet}_{B,p+1}$ could be estimated by solving equation \eqref{wlasso-eq2} using some algorithm to solve Lasso-type problems, e.g. the LARS procedure or Coordinate Descent Optimization, without including an intercept. Then $\hat{\bbet}_{B,1}$ could be solved easily from \eqref{wlasso-eq1}.

The fact that $\mathbf{k}^{*}, \mathbf{y}^{*}$ and $\mathbf{X}^{*\perp T}$ depend on $\hat{\bbet}_{B}$ suggests an iterative procedure, as is usual in robust statistics. Starting from $\hat{\bbet}_{PS}$ we iteratively solve equation \eqref{wlasso-eq2} using the LARS algorithm without including an intercept and then solve for the intercept in \eqref{wlasso-eq1}. Call $\bbet^{(i)}$ the estimate at the $i$-th iteration. Convergence is declared when 
\begin{equation}
\frac{\Vert\bbet^{(i+1)}-\bbet^{(i)} \Vert }{\Vert \bbet^{(i)} \Vert}\leq \delta,
\nonumber
\end{equation}
where $\delta$ is some fixed tolerance parameter. In our simulations we took $\delta=10^{-4}$.

Regarding the calculation of adaptive MM-Lasso estimators, we note that solving \eqref{p-AMM} is equivalent to solving
\begin{equation}
\hat{\bbet}=\arg\min_{\bbet\in\mathbb{R}^{p}%
}\sum\limits_{i=1}^{n}\rho_{1}\left(  \frac{y_i-\hat{\mathbf{x}}^{T}_{i} \bbet}{s_{n}}\right)+\iota_n\Vert\boldsymbol{\beta}\Vert_{1},
\nonumber
\end{equation}
where $\hat{\mathbf{x}}_{i,j}=\mathbf{x}_{i,j} \vert \hat{\beta}_{2,j} \vert^{\varsigma}$ for $j=1,...,p$ and taking $\hat{\beta}_{A,j}=\hat{\beta}_{j} \vert \hat{\beta}_{2,j} \vert^{\varsigma}$. Hence, our procedure to calculate MM-Lasso estimators can be used to calculate adaptive MM-Lasso estimators, simply applying the routine to the data with weighed carriers.
To calculate our proposed adaptive MM-Lasso estimator, we take $\hat{\bbet}_{1}=\hat{\bbet}_{PS}$,  $\hat{\bbet}_{2}=\hat{\bbet}_{B}$ and $\varsigma=1$.

In practice, we chose the $\rho$-functions used to calculate the initial S-Ridge estimator, the MM-Lasso estimator and the adaptive MM-Lasso estimator of the form $\rho_0=\rho_{c_0}^{B}$ and $\rho_1=\rho_{c_1}^{B}$ where $c_1 \geq c_0$ and $\rho_{c}^{B}$ is as in \eqref{rho-marina}. The tuning constants $c_0$ and $c_1$ are chosen as in \cite{Maronna}.

The penalization parameter for $\hat{\bbet}_{B}$, $\lambda_n$, is chosen over a set of candidates via robust 5-fold cross validation, using a $\tau$-scale of the residuals as the objective function. The $\tau$-scale was introduced by \cite{tau 88} to measure in a robust and efficient way the largeness of the residuals
in a regression model. The set of candidate lambdas is taken as 30 equally spaced points between 0 and $\lambda_{max}$, where $\lambda_{max}$ is approximately the minimum $\lambda$ such that all the coefficients of $\hat{\bbet}_{B}$ except the intercept are zero. To estimate $\lambda_{max}$ we first robustly estimate the maximal correlation between $\mathbf{y}$ and the columns of $\mathbf{X}$ using bivariate winsorization as advocated by \cite{RLARS}. We use this estimate as an initial guess for  $\lambda_{max}$ and then improve it using a binary search. If $p>n$, then 0 is excluded from the candidate set. The penalization parameter for $\hat{\bbet}_{A}$, $\iota_n$, is chosen using the same scheme used to choose $\lambda_n$.

The initial S-Ridge estimate is calculated using our own adaption of Maronna's MATLAB code to C++. To solve equation \eqref{wlasso-eq2} we use the FastLasso() function from the {\tt robustHD R} package (\citet{robustHD}). We use the {\tt foreach R} (\citet{foreach}) package for parallel computations when it comes to finding optimal penalization parameters via cross-validation. This provided a significant reduction in computing times in computers with several cores. Extensive parts of our computer code are written in C++ and interfaced with {\tt R} using the {\tt RcppArmadillo} package (\citet{Arma}). An {\tt R} package that includes the functions to calculate the estimators we propose is available at \url{http://esmucler.github.io/mmlasso/}.

\section{Simulations}
\label{sec-sim}

In this section, we compare the performance with regards to prediction accuracy and variable selection properties of 

\begin{itemize}
\item The MM-Lasso estimator described in the previous section.
\item The adaptive MM-Lasso estimator described in the previous section.
\item The Sparse-LTS of \citet{Croux}. The penalization parameter for this estimator is chosen using a BIC-type criterion as advocated by the authors. The estimator was calculated using the sparseLTS() function from the {\tt robustHD R} package.
\item The LS-Lasso estimator. The penalization parameter for this estimator was chosen using 5-fold cross validation using the sum of the squared residuals as the objective function. The estimator was calculated using the lars() function from the {\tt lars R} package (\citet{lars-pack}).
\item The adaptive LS-Lasso estimator. The weights used were the reciprocal of an initial LS-Lasso estimator, calculated as above. Both the initial and the final penalization parameters were chosen using 5-fold cross validation using the sum of the squared residuals as the objective function. The estimator was calculated using the lars() function from the {\tt lars R} package.

\item The Maximum Likelihood Oracle estimator, that is, the Maximum Likelihood estimator applied to the relevant carriers only. When the errors follow a normal distribution, this is the Least Squares estimators applied to the relevant carriers only. Note that in any case, this is not a feasible estimator, and is included for benchmarking purposes only.

\item For the contaminated scenarios, we will also include the Oracle MM estimator: an MM-estimator, calculated with Tukey's bisquare function and tuned to have $85\%$ normal efficiency, applied to the relevant carriers only. The estimator was calculated using the lmRob() function from the {\tt robust R} package (\citet{robust}). Once again, note that this is not a feasible estimator, and is included for benchmarking purposes only.
\end{itemize}
\subsection{Scenarios}
\label{simulation scenarios}

To evaluate the estimators we generate two independent samples of size $n$ of the model $y=\mathbf{x}^{T}\bbet_0+u$. The first sample, called the training sample, is used to fit the estimates and the second sample, called the testing sample, is used to evaluate the prediction accuracy of the estimates. We considered three possible distributions for the errors: a zero mean normal distribution, Student's t-distribution with three degrees of freedom ($t(3)$) and Student's t-distribution with one degree of freedom ($t(1)$). The first case corresponds to the classical scenario of normal errors, the second case has heavy-tailed errors and the third case has extremely heavy-tailed errors. For the first two cases we use the prediction root mean squared error (RMSE) to evaluate the prediction accuracy of the estimates. For the third case, since Student's t-distribution with one degree of freedom does not have a finite first moment, we use the median of the absolute value (MAD) of the prediction  residuals as a measure of the the estimators prediction accuracy. We also evaluate the variable selection performance of the estimators by calculating the false negative ratio (FNR), that is, the fraction of coefficients erroneously set to zero, and the false positive ratio (FPR), the fraction of coefficient erroneously not set to zero.

We consider the following five scenarios for the sample size, the number of covariates, $\bbet_0$ and the distribution of the carriers.

\begin{enumerate}

\item We take $p=8$, $n=40$ and $\bbet_0$ given by: component 1 is 3, component 2 is 1.5, component 6 is 2 and the rest of the coordinates are set to zero. We take $\mathbf{x}\backsim \mathit{N}_{p}(\mathbf{0},\boldsymbol\Sigma)$ with $\Sigma_{i,j}=\rho^{\vert i-j \vert}$ with $\rho =0.5$. For the case of normally distributed errors, we take the standard deviation of the errors to be $\sigma=3$.
\item The same as the last one, but with $n=60$ and $\sigma=1$.
\item We take $p=30$, $n=100$ and $\bbet_0$ given by: components 1-5 are 2.5, components 6-10 are 1.5, components 11-15 are 0.5 and the rest are zero. We take $\mathbf{x}\backsim \mathit{N}_{p}(\mathbf{0},\boldsymbol\Sigma)$ with $\Sigma_{i,j}=\rho^{\vert i-j \vert}$ with $\rho =0.95$. For the case of normally distributed errors, we take the standard deviation of the errors to be $\sigma=1.5$.
\item We take $p=200$, $n=100$ and $\bbet_0$ given by: components 1-5 are 2.5, components 6-10 are 1.5, components 11-15 are 0.5 and the rest are zero. The first $15$ covariates $(x_1,...,x_{15})$ and the remaining $185$ covariates $(x_{16},...,x_{200})$ are independent. The first 15 covariates have a zero mean multivariate normal distribution. The pairwise correlation between the $i$th and $j$th components of $(x_1,...,x_{15})$ is $\rho^{\vert i-j \vert}$ with $\rho=0.5$ for $i, j= 1,...,15$. The final 185 covariates have a zero mean multivariate normal distribution. The pairwise correlation between the $i$th and $j$th components of $(x_{16},...,x_{200})$ is $\rho^{\vert i-j \vert}$ with $\rho=0.5$ for $i, j= 16,...,200$. For the case of normally distributed errors, we take the standard deviation of the errors to be $\sigma=1.5$.
\item The same as the last one, but with $\rho=0.95$.
\item The same as Scenario 1, but with $p=250$ and $n=50$.

\end{enumerate}

In Scenario 1 we have a moderately high $p/n$ ratio. In Scenario 2 we have a relatively low $p/n$ ratio. In Scenario 3 we have $p<n$ and high $p/n$ ratio and in Scenarios 4, 5 and 6 we have $p>n$. Scenarios 1 and 2 were analysed in \cite{lasso} and \cite{Fan}. Scenarios 3, 4 and 5 were analysed in \cite{Huang}.

To evaluate the robustness of the estimators for the case of high-leverage outliers, we introduce contaminations in all six scenarios, for the case of normal errors. Note that we only contaminate the training sample and not the testing sample. We take $m=[0.1 n]$ and for $i=1,..,m$ we set $y_i=5 y_0$ and $\mathbf{x}_i=(5,...,0)$. We moved $y_0$ in an uniformly spaced grid between 0 and 3 with step 0.1 and then between 3 and 10 with step 1. To summarize the results for the contaminated scenarios we report for each estimator the maximum RMSE, FNR and FPR over all outlier sizes $y_0$. We note that the RMSEs of the LS-Lasso and the adaptive LS-Lasso are unbounded as a function of the outlier size and thus the range of outlier sizes considered aims at finding the maximum RMSE of the MM-Lasso, the adaptive MM-Lasso and the Sparse-LTS.

The number of Montecarlo replications for the uncontaminated scenarios was $M=500$. The number of Montecarlo replications for contaminated scenarios was reduced to $M=100$, to keep computation times reasonably low. 

\newpage

\subsection{Results}

We now present the results of our simulation study. All results are rounded to two decimal places. Table \ref{tab:sincon} shows the results for Scenarios 1 through 6 without contamination.

Regarding the prediction accuracy of the estimators, for the case of normal errors, the MM-Lasso and the adaptive MM-Lasso have a RMSE of the same order, and at times even lower than that of Lasso and the adaptive Lasso.  The Sparse-LTS shows a good behaviour in Scenarios 1, 2 and 5, but its RMSE is much larger than that of the other estimators for the remaining scenarios. For the case of errors with t(3) or t(1) distribution, the MM-Lasso and the adaptive MM-Lasso show the best overall performance. We were surprised by the fact that for t(3) errors, the Lasso and the adaptive Lasso have a reasonably low RMSE when compared with the maximum likelihood oracle. As expected, the Lasso and the adaptive Lasso lose all predictive power when the errors have a t(1) distribution. Except for Scenarios 3, 4 and 6, the Sparse-LTS shows a reasonably good performance.
Regarding the variable selection properties of the estimators, we note that the FPR and the FNR of the MM-Lasso are comparable to that of the Lasso, and the FPR and FNR of the adaptive MM-Lasso are comparable to that of the adaptive Lasso for the case of normal errors. For errors with t(3) or t(1) distribution, the MM-Lasso and the adaptive MM-Lasso generally show the best behaviour. The FPR of the adaptive MM-Lasso is lower than that of the MM-Lasso, but the price to pay for this improvement is an increase in the FNR. Note that for Scenarios 1, 2 and 3 the Sparse-LTS has a rather high FPR, always greater than 0.5.

In Table \ref{tab:con} we show the results for Scenarios 1 through 6 under high-leverage contamination. The MM-Lasso and the adaptive MM-Lasso show the best overall behaviour. The Sparse-LTS shows a good behaviour for Scenarios 1, 2, and the best behaviour for Scenario 5, but its maximum RMSE is much larger than that of the MM-Lasso and the adaptive MM-Lasso for the rest of the scenarios. As expected, the maximum RMSE of the Lasso and the adaptive Lasso is very large in all cases. In Figure \ref{graf:curvahuang} we show the RMSEs of the estimators as a function of the outlier size for Scenario 3. The MM-Lasso has the overall best behaviour, followed closely by the adaptive MM-Lasso. Note that the RMSE curves of the Lasso and of the adaptive Lasso are unbounded as a function of the outlier size: by taking larger outlier sizes the maximum RMSEs of the MM-Lasso, the adaptive MM-Lasso and the Sparse-LTS would not change, but those of the Lasso and the adaptive Lasso would increase without bound. Regarding the variable selection properties of the estimators, the MM-Lasso and the adaptive MM-Lasso show the best overall balance between a low FNR and a low FPR. Note that for Scenarios 1, 2, 3 the maximum FPR of the Sparse-LTS is very high.

\begin{center}
\begin{figure}[!h]
  \centering
  \includegraphics[width = \textwidth]{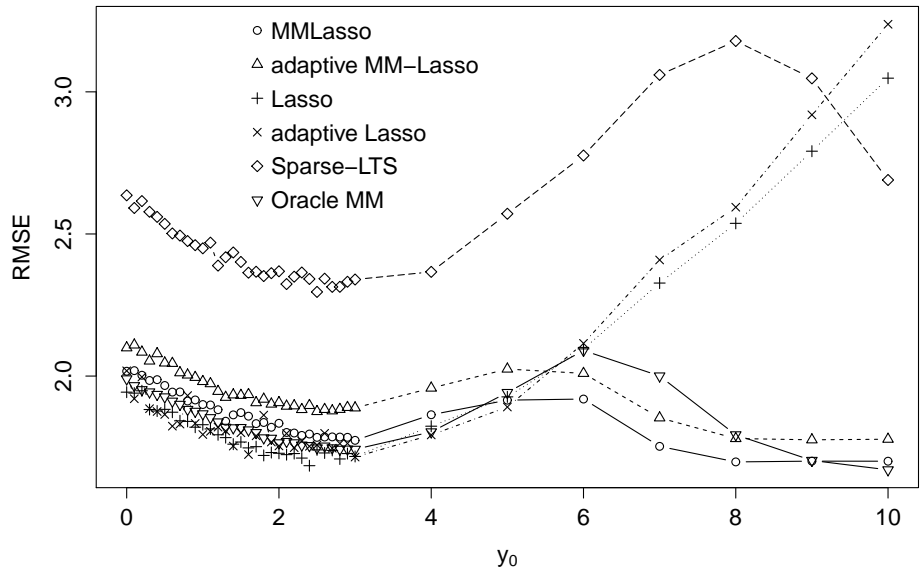}
  \caption{RMSEs as a function of outlier sizes for each of the estimators for the third scenario, with $p=30$, $n=100$, normal errors and 10\% contamination. RMSEs are averaged over 100 replications.}
  \label{graf:curvahuang}
\end{figure}
\end{center}

\newpage

\begin{table}[!htp]
  \centering
\tabcolsep=0.11cm  
  \begin{tabular}{p{3cm}p{1cm}p{0.7cm}p{0.7cm}p{1cm}p{0.7cm}p{0.7cm}p{1cm}p{0.7cm}p{0.7cm}}
    \hline\noalign{\smallskip}
	Scenario & Normal && & $t(3)$ &&&$t(1)$ &&\\
    \hline\noalign{\smallskip}
    	 &RMSE & FNR & FPR &RMSE & FNR & FPR&MAD & FNR & FPR\\
        \hline
		 		1 $(n,p)=(40,8)$  &\\
    \hline\noalign{\smallskip}
	 MM-Lasso & 3.42 & 0.04 & 0.52 & 1.77 & 0 & 0.52& 1.36 & 0.01 & 0.50\\
 	adaptive MM-Lasso & 3.43 & 0.09 & 0.27 & 1.75 & 0 & 0.20& 1.32 & 0.02 & 0.21\\
    Sparse-LTS & 3.92 & 0.03 & 0.82 & 1.91 & 0 & 0.85& 1.44 & 0 & 0.69\\
        Lasso & 3.33 & 0.02 & 0.43 & 1.84 & 0 & 0.46& 9.9 & 0.38 & 0.28\\
        adaptive Lasso & 3.28 & 0.06 & 0.26 & 1.82 & 0.01 & 0.29& 10 & 0.46 & 0.19\\
            Oracle &3.15& 0 & 0   & 1.69 & 0 & 0& 1.16 & 0 & 0\\
        \hline
		 		2 $(n,p)=(60,8)$ &\\
    \hline\noalign{\smallskip}
    	 MM-Lasso & 1.09 & 0 & 0.53	 & 1.77 & 0 & 0.51& 1.21 & 0 & 0.46\\
 	adaptive MM-Lasso & 1.07 & 0 & 0.21 & 1.75 & 0 & 0.18& 1.18 & 0 & 0.16\\
    Sparse-LTS & 1.16 & 0 & 0.69 & 1.76 & 0 & 0.67& 1.24 & 0 & 0.57\\
       Lasso & 1.07 & 0 & 0.48 & 1.82 & 0 & 0.46& 5.38 & 0.39 & 0.29\\
    adaptive Lasso & 1.07 & 0 & 0.31 & 1.73 & 0 & 0.31& 5.54 & 0.47 & 0.19\\
            Oracle &1.04& 0 & 0   & 1.72& 0 & 0& 1.10 & 0 & 0\\
        \hline
		 		3 $(n,p)=(30,100)$ &\\
    \hline\noalign{\smallskip}
  	 MM-Lasso & 1.69 & 0.13 & 0.21 & 1.75 & 0.10 & 0.27& 1.28 & 0.15 & 0.17\\
 	adaptive MM-Lasso & 1.77 & 0.26 & 0.09 & 1.80 & 0.21 & 0.09& 1.39 & 0.29 & 0.06\\
    Sparse-LTS & 2.25 & 0 & 1 & 2.14 & 0 & 1& 1.78 & 0.01 & 0.97\\
       Lasso & 1.74 & 0.11 & 0.22 & 1.90 & 0.12 & 0.27& 10.7 & 0.55 & 0.21\\
    adaptive Lasso & 1.74 & 0.21 & 0.13 & 1.94 & 0.22 & 0.14& 10.7 & 0.72 & 0.09\\
            Oracle &1.63& 0 & 0   & 1.73 & 0 & 0& 1.27 & 0 & 0\\
        \hline
		 		4 $(n,p)=(100,200)$ &\\
    \hline\noalign{\smallskip}
  	 MM-Lasso & 1.90 & 0.02 & 0.12 & 2.02 & 0.02 & 0.1& 1.90 & 0.08 & 0.09\\
 	adaptive MM-Lasso & 1.78 & 0.08 & 0.02 & 1.89 & 0.07 & 0.01& 1.71 & 0.15 & 0.03\\
    Sparse-LTS & 3.32 & 0.14 & 0.11 & 3.17 & 0.12 & 0.01& 2.43 & 0.13 & 0.12\\
       Lasso & 1.96 & 0.02 & 0.23 & 2.19 & 0.02 & 0.23& 7.79 & 0.51 & 0.1\\
    adaptive Lasso & 2.16 & 0.04 & 0.15 & 2.41 & 0.05 & 0.16& 9.10& 0.56 & 0.07\\
            Oracle &1.64& 0 & 0   & 1.77 & 0 & 0& 1.28 & 0 & 0\\
        \hline
		 		5 $(n,p)=(100,200)$ &\\
    \hline\noalign{\smallskip}
  	 MM-Lasso & 1.92 & 0.16 & 0.08 & 1.89 &0.11  &0.06 & 1.42& 0.16 & 0.05\\
 	adaptive MM-Lasso & 1.94 & 0.29 & 0.03 & 1.89 & 0.22& 0.02& 1.47 & 0.31 & 0.01\\
    Sparse-LTS & 1.89 & 0.13 & 0 & 1.98& 0.11 &0& 1.47 & 0.15 & 0\\
       Lasso & 1.88 & 0.11 & 0.21 & 2.12 & 0.12 & 0.22& 6.33 & 0.57 & 0.10\\
    adaptive Lasso & 2.06 & 0.18 & 0.13 & 2.29 & 0.19 & 0.14& 6.75 & 0.70 & 0.10\\
                Oracle &1.64& 0 & 0   & 1.77 & 0 & 0& 1.29 & 0 & 0\\
 \hline
		 		6 $(n,p)=(50,250)$ &\\
    \hline\noalign{\smallskip}
  	 MM-Lasso &4.05 & 0.12 & 0.07 & 1.99 & 0 & 0.06& 2.05 & 0.08 & 0.05\\
 	adaptive MM-Lasso & 3.99 & 0.18 & 0.03 & 1.80 & 0.01 & 0.01& 1.79 & 0.12 & 0.02\\
    Sparse-LTS & 4.72 & 0.26 & 0.12 & 2.60 & 0.04 & 0.09& 2.22 & 0.07 & 0.11\\
       Lasso & 3.67 & 0.05 & 0.07 & 2.04 & 0.01 & 0.07& 30.5 & 0.62 & 0.03\\
    adaptive Lasso & 3.97 & 0.06 & 0.06 & 2.26 & 0.01 & 0.06& 31.3 & 0.64 & 0.02\\
            Oracle &3.13& 0 & 0   & 1.67 & 0 & 0& 1.12 & 0 & 0\\
     \noalign{\smallskip}\hline\noalign{\smallskip}
  \end{tabular}
  \caption{Results for all the simulation scenarios, with normal, $t(3)$ and $t(1)$ distributed errors. RMSE, MAD, FNR and FPR, averaged over 500 replications are reported for each estimator.}
    \label{tab:sincon}
\end{table}

\FloatBarrier

\newpage

\begin{table}[!htp]
  \centering
  \begin{tabular}{p{3cm}p{2cm}p{1.5cm}p{1.5cm}}
    \hline\noalign{\smallskip}
    Scenario &Max. RMSE &Max. FNR &Max. FPR \\
        \hline
		 		1 $(n,p)=(40,8)$ &\\
    \hline\noalign{\smallskip}
    MM-Lasso & 4.38 & 0.11 & 0.57 \\
    adaptive MM-Lasso & 4.43 & 0.25 & 0.32 \\
    Sparse-LTS & 4.92 & 0.07 & 0.95 \\
        Lasso & 5.78 & 0.27 & 0.49 \\
    adaptive Lasso & 6.14 & 0.36 & 0.33 \\
    Oracle MM& 3.71 & 0 & 0 \\
        \hline
		 		2 $(n,p)=(60,8)$ &\\
    \hline\noalign{\smallskip}
        MM-Lasso & 1.39 & 0 & 0.59 \\
    adaptive MM-Lasso & 1.38 & 0.01 & 0.36 \\
    Sparse-LTS & 1.42 & 0 & 0.92 \\
        Lasso & 4.89 & 0.19 & 0.56 \\
    adaptive Lasso & 5.13 & 0.25 & 0.38 \\
    Oracle MM& 1.21 & 0 & 0 \\
        \hline
		 		3 $(n,p)=(100,30)$ &\\
    \hline\noalign{\smallskip}       
 MM-Lasso & 2.02 & 0.20 & 0.35 \\
    adaptive MM-Lasso & 2.11 & 0.36 & 0.21 \\
    Sparse-LTS & 3.18 & 0 & 1 \\
        Lasso & 3.05 & 0.25 & 0.26 \\
    adaptive Lasso & 3.24 & 0.41 & 0.15 \\
    Oracle MM&2.09& 0 & 0   \\
        \hline
		 		4 $(n,p)=(100,200)$ &\\
    \hline\noalign{\smallskip}   
    MM-Lasso & 4.14 & 0.13 & 0.24 \\
    adaptive MM-Lasso & 4.02 & 0.21 & 0.12 \\
    Sparse-LTS & 5.25 & 0.28 &  0.15\\
        Lasso & 6.74 & 0.31 & 0.21 \\
    adaptive Lasso & 7.96 & 0.40 & 0.14 \\
                Oracle MM&2.09& 0 & 0   \\
        \hline
		 		5 $(n,p)=(100,200)$&\\
    \hline\noalign{\smallskip}  
    MM-Lasso & 2.48 & 0.21 & 0.15 \\
    adaptive MM-Lasso & 2.72 & 0.37 & 0.05 \\
    Sparse-LTS & 2.14 & 0.22 & 0\\
        Lasso & 20.25 & 0.64 & 0.15 \\
    adaptive Lasso & 13.03 & 0.79& 0.06 \\
                Oracle MM&2.09& 0 & 0   \\
         \hline
		 		6 $(n,p)=(50,250)$ &\\
    \hline\noalign{\smallskip}   
    MM-Lasso & 4.97 & 0.36 & 0.08 \\
    adaptive MM-Lasso & 5.08 & 0.45 & 0.04 \\
    Sparse-LTS & 5.40 & 0.47 &  0.11\\
        Lasso & 6.04 & 0.42& 0.07 \\
    adaptive Lasso & 7.89 & 0.45 & 0.06 \\
                Oracle MM&3.68& 0 & 0   \\
    \noalign{\smallskip}\hline\noalign{\smallskip}

  \end{tabular}
  \caption{Results for all the scenarios with normal errors and 10\% contaminated observations. Maximum RMSEs, FNRs and FPRs over all outlier sizes are averaged over 100 replications.}
  \label{tab:con}
\end{table}

Finally, we calculated the computing times of the adaptive MM-Lasso, the MM-Lasso and the Sparse-LTS for several of the considered scenarios, for the case of normal errors and no contamination. Since the computing times for the adaptive MM-Lasso and the MM-Lasso were very similar, we only report the results for the adaptive MM-Lasso. Computing times were averaged over 5 replications and calculations were performed on {\tt R} 3.0.2 on a 3.07x4 GHz Intel Core i7 PC. We see that in Scenarios 1, 3 and 5 the Sparse-LTS is considerably faster than the adaptive MM-Lasso. However, in Scenario 6, the adaptive MM-Lasso is 3 times faster than the Sparse-LTS.

\begin{table}[!htp]
  \centering
  \begin{tabular}{p{3cm}p{3cm}p{3cm}}
    \hline\noalign{\smallskip}
     Scenario&adaptive MM-Lasso & Sparse-LTS \\
    \hline\noalign{\smallskip}
     1 $(n,p)=(40,8)$&3.33&0.7 \\
  	 3 $(n,p)=(100,30)$&7.35&1.71 \\
     5 $(n,p)=(100,200)$&41.75&28.51 \\
     6 $(n,p)=(50,250)$&8.05&25.89 \\    
    \noalign{\smallskip}\hline\noalign{\smallskip}

  \end{tabular}
  \caption{Computing times in seconds for the adaptive MM-Lasso and the Sparse-LTS, averaged over 5 replications.}
  \label{tab:times}
\end{table}

\FloatBarrier

\newpage

\section{A real high-dimensional data set}
\label{sec-real}

In this section, we analyse a data set corresponding to electron-probe X-ray microanalysis of archaeological glass vessels, where each of $n=180$ glass vessels is represented by a spectrum on 1920 frequencies. For each vessel the contents of thirteen chemical compounds are registered. This data set appears in \citet{Janssens}, and was previously analysed in \cite{Maronna}. We fit a linear model where the response variable is the content of the $13th$ chemical compound (PbO) and the carriers are the 1920 frequencies measures on each glass vessel. Since for frequencies below 15 and above 500 the values of $x_{ij}$ are almost null and show very little variability, we keep frequencies 15 to 500, so that we have $p = 486$. We apply the MM-Lasso, the adaptive MM-Lasso, the Sparse-LTS, the Lasso and the adaptive Lasso estimators to the data.

The MM-Lasso selects seven variables: the $28th$, $145th$, $337th$, $338th$, $372nd$, $374th$ and $403rd$ frequencies. The adaptive MM-Lasso selects four variables: the $28th$, $145th$, $337th$ and $374th$ frequencies. Thus, the adaptive MM-Lasso drops three of the variables selected by the MM-Lasso. The Sparse-LTS selects three variables: the  $338th$ and $403rd$ and $466$ frequencies. The Lasso selects 70 variables, the adaptive Lasso selects 49. Hence, all three robust estimators produce models that are sparser and easier to interpret.

To asses the prediction accuracy of the estimators, we used 5-fold cross-validation. The criterion used was a $\tau$-scale of the residuals, calculated as in \citet{OGK}. The adaptive MM-Lasso and the Lasso show the best behaviour by far, followed by the Lasso, the adaptive Lasso and the Sparse-LTS, in that order.

%
%
%
%
%
\begin{table}[!ht]
  \centering
  \begin{tabular}{p{3cm}p{2cm}p{1.5cm}}
    \hline\noalign{\smallskip}
    &$\tau$-scale\\
    \noalign{\smallskip}\hline\noalign{\smallskip}
    MM-Lasso & 0.086 \\
    adaptive MM-Lasso & 0.083\\
    Sparse-LTS & 0.329\\
    Lasso &0.131\\
    adaptive Lasso &0.138\\	        

    \noalign{\smallskip}\hline\noalign{\smallskip}
  \end{tabular}
  \caption{Cross-validated $\tau$-scale of the residuals of each of the estimators for the electron-probe X-ray microanalysis data.}
  \label{tab:vessel}
\end{table}

%
%
%
%
%
%
%
%
%
%
%
%
\FloatBarrier

\section{Conclusions}
\label{conclusions}

We have studied the robust and asymptotic properties of MM-Bridge and adaptive MM-Bridge regression estimators. We proved that, for the case of a fixed number of covariates, MM-Bridge estimators can have the \textit{oracle property} defined in \cite{Fan} whenever $q<1$. We proved that adaptive MM-Bridge estimators can have the \textit{oracle property} for all $t\leq1$. We also derived the asymptotic distribution of the MM-Ridge estimator of \citet{Maronna}. 

We proposed an algorithm to calculate both the MM-Lasso and the adaptive MM-Lasso. Our simulation study suggests that, at least for the scenarios considered, the proposed MM-Lasso and adaptive MM-Lasso estimators provide the best balance between prediction accuracy and sparse modelling for uncontaminated samples, and stability in the presence of outliers. The adaptive MM-Lasso reduces the false positive ratio of the MM-Lasso, with the unpleasant, and foreseeable, side effect of an increase in the false negative ratio. We note that even though we derived our asymptotic results for the case of a fixed number of covariates, the MM-Lasso and adaptive MM-Lasso estimators can be calculated for $p>n$. The study of the asymptotic properties of the these estimators for regression models with a diverging number of parameters is part of our future work.

\appendix
\section{Appendix}
\label{Appendix}

\begin{proof}[Proof of Theorem \ref{Theo-break-br}]
Take $C\subset\{1,2,...,n\}$ such that $\#C=n-1$ and a sequence $(\mathbf{x}^{T}_{Ni},y_{Ni})_{N\in\mathbb{N}}$, such that $(\mathbf{x}_{N,i}^{\text{T}%
},y_{N,i})=(\mathbf{x}_{i}^{T},y_{i})$ for $i\not \in C$ and all
$N\in\mathbb{N}$. Let $\widehat{\bbet}_{B}^{N}$ and $\widehat{\bbet}_{1}^{N}$denote the estimators $\widehat{\bbet}_{B}$ and $\widehat{\bbet}_{1}$ computed in $(\mathbf{x}^{T}_{Ni},y_{Ni})_{N\in\mathbb{N}}$. Since there are a finite number of sets
included in $\{1,...,n\}$, to prove the theorem it will be enough to show that
$(\widehat{\bbet}_{B}^{N})_{N}$ is bounded. Suppose that this is
not so, then eventually passing to a subsequence we can assume that
$\Vert\widehat{\bbet}_{B}^{N}\Vert\rightarrow\infty$ when
$N\rightarrow\infty$. Since $\rho_1$ is bounded, for sufficiently large $N$ we have that
\begin{equation}
\sum\limits_{i=1}^{n}\rho_{1}\left(  \frac{r_{i}(\widehat{\bbet}_{B}^{N})}{s_{n}%
(\mathbf{r}(\widehat{\bbet}_{1}^{N}))}\right)+\lambda_n\Vert \widehat{\bbet}_{B}^{N} \Vert_{q}^{q} > 
\sum\limits_{i=1}^{n}\rho_{1}\left(  \frac{r_{i}(\mathbf{0})}{s_{n}%
(\mathbf{r}(\widehat{\bbet}_{1}^{N}))}\right)+\lambda_n\Vert \mathbf{0} \Vert_{q}^{q},
\nonumber
\end{equation}
which contradicts the definition of $\widehat{\bbet}_{B}^{N}$.
\end{proof}

\begin{proof}[Proof of Theorem \ref{Theo-break-ad}]
Let $m=nFBP(\widehat{\bbet}_2)$. Take $C\subset\{1,2,...,n\}$ such that $\#C\leq m$ and a sequence $(\mathbf{x}^{T}_{Ni},y_{Ni})_{N\in\mathbb{N}}$, such that $(\mathbf{x}_{N,i}^{\text{T}%
},y_{N,i})=(\mathbf{x}_{i}^{T},y_{i})$ for $i\not \in C$ and all
$N\in\mathbb{N}$. Let $\widehat{\bbet}_{A}^{N}$, $\widehat{\bbet}_{2}^{N}$ and $\widehat{\bbet}_{1}^{N}$ denote the estimators $\widehat{\bbet}_{A}$, $\widehat{\bbet}_{2}$ and $\widehat{\bbet}_{1}$ computed in $(\mathbf{x}^{T}_{Ni},y_{Ni})_{N\in\mathbb{N}}$. Note that since $\#C\leq m$, $\widehat{\bbet}_{2}^{N}$ is bounded. Since there are a finite number of sets
included in $\{1,...,n\}$, to prove the theorem it will be enough to show that
$(\widehat{\bbet}_{A}^{N})_{N}$ is bounded. Suppose that this is
not so, then eventually passing to a subsequence we can assume that for some $j_0$, $\vert\widehat{\bbet}_{A,j_0}^{N}\vert\rightarrow\infty$ when
$N\rightarrow\infty$. Hence, there exists $N_0$, such that for $N\geq N_0$, $\widehat{\beta}_{2,j_0}^{N}\neq 0$. It follows that $\vert\widehat{\bbet}_{A,j_0}^{N}\vert^{t} / \vert \widehat{\beta}_{2,j_0}^{N} \vert^{\varsigma} \rightarrow
\infty$.

Since $\rho_1$ is bounded, for sufficiently large $N$ we have that
\begin{equation}
\sum\limits_{i=1}^{n}\rho_{1}\left(  \frac{r_{i}(\widehat{\bbet}_{A}^{N})}{s_{n}%
(\mathbf{r}(\widehat{\bbet}_{1}^{N}))}\right)+\iota_n \sum\limits_{i=1}^{p} \frac{\vert \widehat{\beta}_{A,j}^{N} \vert^{t}}{\vert \widehat{\beta}_{2,j}^{N} \vert^{\varsigma}} > 
\sum\limits_{i=1}^{n}\rho_{1}\left(  \frac{r_{i}(\mathbf{0})}{s_{n}%
(\mathbf{r}(\widehat{\bbet}_{1}^{N}))}\right)+\iota_n \sum\limits_{i=1}^{p} \frac{\vert 0 \vert^{t}}{\vert \widehat{\beta}_{2,j}^{N} \vert^{\varsigma}}
\nonumber
\end{equation}
which contradicts the definition of $\widehat{\bbet}_{A}^{N}$.
\end{proof}

Define for $\boldsymbol{\beta}\in\mathbb{R}^{p}$ $s(\boldsymbol{\beta})$ by
\begin{equation}
E_{H_0}\rho_0\left(\frac{y-\mathbf{x}^T\bbet}{s(\bbet)}\right)=b,
\nonumber
\end{equation}
and let
\begin{equation}
g(\bbet)=E_{H_0}\rho_1\left(\frac{y-\mathbf{x}^T\bbet}{s(\bbet_{0})}\right).
\nonumber
\end{equation}

It can be readily verified that $s(\bbet)$ is continuous and positive. Lemma 4.2 of \cite{tau 86} shows that $s(\bbet)$ has a unique minimum at $\bbet=\boldsymbol{\beta_0}$, and hence proves the Fisher consistency of S-estimators of regression. Theorem 6 of \cite{Fasano}, shows that $g(\bbet)$ has a unique minimum at $\bbet=\bbet_{0}$, and  hence proves the Fisher consistency of MM-estimators of regression.

The following Lemma, which appears in \cite{tau 86} as Lemma 4.5, is a key result.

\begin{lemma}
\label{Lemma S unif}
Let $(\mathbf{x}_{i}^{\text{T}},y_{i})$, $i=1,...,n,$ be i.i.d observations
with distribution $H_{0}$, which satisfies \eqref{H0}. Assume [B1]-[B3] hold. Let $K\subseteq\mathbb{R}^{p}$ be a compact set. Then
\begin{equation}
\sup_{\boldsymbol{\beta}\in K%
} \vert s_n(\mathbf{r}(\boldsymbol{\beta})) - s(\boldsymbol{\beta}) \vert \cas 0
\nonumber
\end{equation}
\end{lemma}

To ease notation, we will henceforth note $s_{n}=s_{n}(\mathbf{r}(\widehat{\bbet}_{1}))$, where $\widehat{\bbet}_{1}$ is as in \eqref{p-MM}. 

\begin{proof}[Proof of Theorem \ref{Cons}]
We first prove (i). Let
\begin{equation}
Z_n^1(\bbet)=s_{n}^2(\mathbf{r}(\bbet))+\frac{\gamma_n}{n}\Vert\bbet\Vert_{r}^{r},
\nonumber
\end{equation}
so that $\arg\min_{\bbet\in\mathbb{R}^{p}%
} Z_n^1(\bbet)=\widehat{\bbet}_{PS}$.
To prove (i), it suffices to show that
\begin{equation}
\widehat{\bbet}_{PS}\text{ is bounded with probability 1}
\label{S bounded}
\end{equation}
and that given a compact set $K$, we have that
\begin{equation}
\sup_{\bbet\in K%
} \vert Z_n^1(\bbet) - s^{2}(\bbet) \vert \cas 0.
\label{uniform conv s}
\end{equation}

Theorem 4.1 of \cite{tau 88} shows that $\widehat{\bbet}_{S}$ converges almost surely to $\bbet_{0}$ and so \eqref{S bounded} follows from \eqref{norm-ineq-S}. Note that the second term in $Z_n^1$ converges uniformly to zero over compact sets, and hence Lemma \ref{Lemma S unif} and the continuity of $s(\bbet)$ show that \eqref{uniform conv s} holds. Thus \eqref{S-cs} is proved.

Next, we prove (iii). The proof of (ii) is essentially the same, and is thus omitted.

Note that by definition of $\widehat{\bbet}_{A}$
\begin{equation}
\frac{1}{n}\sum\limits_{i=1}^{n} \rho_1 \left( \frac{y_i-\mathbf{x}_{i}^{T} \widehat{\bbet}_{A}}{s_n}  \right)\leq  \frac{1}{n}\sum\limits_{i=1}^{n} \rho_1 \left( \frac{u_i}{s_n}  \right) + \frac{\iota_n}{n}  \sum\limits_{j=1}^{p} \frac{\vert \beta_{0,j} \vert^{t}}{\vert \widehat{\beta}_{2,j} \vert^{\varsigma}}
\label{desig-cons}
\end{equation}
The second term in \eqref{desig-cons} is
\begin{equation}
\leq \frac{\iota_n}{n}  \sum\limits_{j=1}^{s} \frac{\vert \beta_{0,j} \vert^{t}}{\vert \widehat{\beta}_{2,j}\vert^{\varsigma}} = O_{P} \left( \frac{\iota_n}{n} \right),
\nonumber
\end{equation}
since $\widehat{\bbet_2}$ is consistent by assumption. Since $\widehat{\bbet}_{1}$ is consistent by assumption, by Lemma \ref{Lemma S unif}, $s_n \cas s(\bbet_0)$. Thus, the Law of large numbers and the Bounded convergence theorem imply that the right hand side of \eqref{desig-cons} converges almost surely to
\begin{equation}
b^{*} \doteq E_{F_0} \rho_1 \left( \frac{u}{s(\bbet_0)}  \right).
\nonumber
\end{equation}

Hence, 
\begin{equation}
\limsup \frac{1}{n}\sum\limits_{i=1}^{n} \rho_1 \left( \frac{y_i-\mathbf{x}_{i}^{T} \widehat{\bbet}_{A}}{s_n}  \right) \leq b^{*} \text{ a.s.}.
\nonumber
\end{equation}

One can easily show that the graphs of the family of functions
\begin{equation}
\mathcal{H}\doteq \left \{ \rho_1 \left( \frac{y - \mathbf{x}^{T}\mathbf{b}}{s} \right) : \: \mathbf{b}\in\mathbb{R}^{p}, \: s>0 \right \},
\nonumber
\end{equation}
form a VC class of sets with a constant envelope. The proof of this is essentially the same as the one that appears on page 29 of \citet{Pollard}. It follows that $\mathcal{H}$ is a Glivenko-Cantelli class of functions, i.e.
\begin{equation}
\sup_{\mathbf{b}\in\mathbb{R}^{p}, \: s>0} \vert \frac{1}{n}\sum\limits_{i=1}^{n}  \rho_1 \left( \frac{y_{i} - \mathbf{x}_{i}^{T}\mathbf{b}}{s} \right) - E_{H_0}\rho_1 \left( \frac{y - \mathbf{x}^{T}\mathbf{b}}{s} \right) \vert \cas 0.
\label{glivenko}
\end{equation}
Hence, it follows from \eqref{glivenko} and Theorem 6 of \citet{Fasano} that for any $\varepsilon>0$
\begin{equation}
\lim \inf_{\varepsilon \leq \Vert \bbet - \bbet_{0}\Vert } \frac{1}{n}\sum\limits_{i=1}^{n} \rho_1 \left( \frac{y_i-\mathbf{x}_{i}^{T} \bbet}{s_n}  \right) > b^{*} \text{ a.s.}.
\nonumber
\end{equation}

It must be that 
\begin{equation}
\widehat{\boldsymbol{\beta}}_{A} \cas \bbet_{0}.
\nonumber
\end{equation}

\end{proof}

\begin{proof}[Proof of Theorem \ref{Rate}]
We prove (ii), the proof of (i) is essentially the same, but replacing $\iota_n$ for $\lambda_n$ and taking $\varsigma=0$.

Let
\begin{equation}
Z_n^2(\bbet)=\frac{1}{n}\sum\limits_{i=1}^{n}\rho_{1}\left(  \frac{r_{i}(\bbet)}{s_{n}}\right)+\frac{\iota_n}{n} \sum\limits_{j=1}^{p} \frac{\vert \beta_{j} \vert^{t}}{\vert \widehat{\beta}_{2,j} \vert^{\varsigma}},
\nonumber
\end{equation}
so that $\arg\min_{\bbet\in\mathbb{R}^{p}%
} Z_n^2(\bbet)=\widehat{\bbet}_{A}$.

Note that
\begin{align*}
Z_{n}^{2}(\boldsymbol{\beta}_{0})  & =\frac{1}{n}\sum\limits_{i=1}^{n}\rho
_{1}\left(  \frac{u_{i}}{s_{n}}\right)  +\frac{\iota_n}{n} \sum\limits_{j=1}^{s} \frac{\vert \beta_{0,j} \vert^{t}}{\vert \widehat{\beta}_{2,j} \vert^{\varsigma}}.
\end{align*}

A second order Taylor expansion shows that
\begin{align*}
Z_{n}^{2}(\widehat{\boldsymbol{\beta}}_{A})  & =\frac{1}{n}\sum\limits_{i=1}%
^{n}\rho_{1}\left(  \frac{r_{i}(\widehat{\bbet}_{A})}{s_{n}}\right)
+\frac{\iota_n}{n}\sum\limits_{j=1}^{p} \frac{\vert \widehat{\beta}_{A,j} \vert^{t}}{\vert \widehat{\beta}_{2,j} \vert^{\varsigma}}\\
& =\frac{1}{n}\sum\limits_{i=1}^{n}\rho_{1}\left(  \frac{u_{i}}{s_{n}}\right)
-\frac{(\widehat{\boldsymbol{\beta}}_{A}-\boldsymbol{\beta}_{0})^{\text{T}}%
}{ns_{n}}\sum\limits_{i=1}^{n}\psi_{1}\left(  \frac{u_{i}}{s_{n}}\right)
\mathbf{x}_{i}\\
& +\frac{1}{2}\frac{1}{s_{n}^{2}}(\widehat{\boldsymbol{\beta}}_{A}-\boldsymbol{\beta}%
_{0})^{\text{T}}\left(  \frac{1}{n}\sum\limits_{i=1}^{n}\psi_{1}^{\prime}\left(  \frac
{u_{i}-\zeta_{i}\mathbf{x}_{i}^{T}(\widehat{\bbet}_{A}-\bbet%
_{0})}{s_{n}}\right)  \mathbf{x}_{i}\mathbf{x}%
_{i}^{\text{T}}\right)  (\widehat{\boldsymbol{\beta}}_{A}-\boldsymbol{\beta
}_{0})\\&
+\frac{\iota_n}{n}\sum\limits_{j=1}^{p} \frac{\vert \widehat{\beta}_{A,j} \vert^{t}}{\vert \widehat{\beta}_{2,j} \vert^{\varsigma}},
\end{align*}
with $0\leq\zeta_{i}\leq1.$ By Lemma \ref{Lemma S unif}, $s_n \cas s(\bbet_0)$, and hence by Lemma 4.2 of \cite{MM 85}
\[
\ \frac{1}{n}\sum\limits_{i=1}^{n}\psi_{1}^{\prime}\left(  \frac{u_{i}-\zeta_{i}\mathbf{x}_{i}^{T}(\widehat{\bbet}_{A}-\bbet%
_{0})}{s_{n}}\right)  \mathbf{x}_{i}\mathbf{x}_{i}^{\text{T}} \cas E_{F_0}\psi_{1}^{\prime}\left(\frac{u}{s(\bbet_0)}\right)\mathbf{V}_{\mathbf{x}}.
\]

Then
\begin{align*}
A_{n}\doteq \frac{1}{2}\frac{1}{s_{n}^{2}}(\widehat{\boldsymbol{\beta}}_{A}-\boldsymbol{\beta
}_{0})^{\text{T}}\left(  \frac{1}{n}\sum\limits_{i=1}^{n}\psi_{1}^{\prime}\left(  \frac
{u_{i}-\zeta_{i}\mathbf{x}_{i}^{T}(\widehat{\bbet}_{A}-\bbet%
_{0})}{s_{n}}\right)  \mathbf{x}_{i}\mathbf{x}%
_{i}^{\text{T}}\right)  (\widehat{\bbet}_{A}-\bbet_{0}) \\ \geq c_{n} \Vert \widehat{\boldsymbol{\beta}}_{A}-\boldsymbol{\beta}_{0}\Vert^{2},
\end{align*}
where $c_{n} \cas c_{0}>0$.

We also have that by Lemma 5.1 of \cite{MM 85} and the Central Limit Theorem
\[
B_{n}\doteq\frac{1}{\sqrt{n}s_{n}}\sum\limits_{i=1}^{n}\psi_{1}\left(  \frac{u_{i}%
}{s_{n}}\right)  \mathbf{x}_{i}=O_{P}(1).
\]

Put%
\[
C_{n}\doteq   \frac{\iota_n}{n} \sum\limits_{j=1}^{p} \frac{\vert \widehat{\beta}_{A,j} \vert^{t}}{\vert \widehat{\beta}_{2,j} \vert^{\varsigma}} - \frac{\vert \beta_{0,j} \vert^{t}}{\vert \widehat{\beta}_{2,j} \vert^{\varsigma}}.
\]

Then, since $\widehat{\bbet}_{A}$ is strongly consistent for $\bbet_0$ and the first $s$ coordinates of $\bbet_{0}$ are non zero, for large enough $n$ the first $s$ coordinates of $\widehat{\bbet}_{A}$ stay away from zero with arbitrarily high probability. Applying the Mean Value Theorem we get that
\begin{align*}
C_{n} \geq \frac{\iota_n}{n} \sum\limits_{j=1}^{s}  \frac{\vert \widehat{\beta}_{A,j} \vert^{t}}{\vert \widehat{\beta}_{2,j} \vert^{\varsigma}} - \frac{\vert \beta_{0,j} \vert^{t}}{\vert \widehat{\beta}_{2,j}  \vert^{\varsigma}}=\frac{\iota_n}{n}  \sum\limits_{j=1}^{s} t \frac{\vert \theta_{j} \vert^{t-1}}{\vert \widehat{\beta}_{2,j} \vert^{\varsigma}}  (\widehat{\beta}_{A,j} \ - \beta_{0,j}),
\end{align*}
for some $\theta_{j}$ such that $\vert \theta_{j} - \beta_{0,j}\vert \leq \vert \theta_{j} - \widehat{\beta}_{A,j}\vert$. 
Since $\iota_n = O(\sqrt{n})$ and $\widehat{\bbet}_{A}$ and $\widehat{\bbet}_{2}$ are consistent, we have that for some $M>0$, for large enough $n$, with arbitrarily high probability

\begin{equation}
C_n \geq \frac{-M}{\sqrt{n}} \Vert \widehat{\bbet}_{A} -{\bbet}_{0} \Vert.
\nonumber
\end{equation}

Then
\begin{align*}
Z_{n}^{2}(\widehat{\boldsymbol{\beta}}_{A})-Z_{n}^{2}(\boldsymbol{\beta}_{0})
& =A_{n}-\frac{1}{\sqrt{n}}(\widehat{\boldsymbol{\beta}}_{A}-\boldsymbol{\beta
}_{0})^{\text{T}}B_{n}+C_{n}\\
& \geq c_{n}\Vert(\widehat{\boldsymbol{\beta}}_{A}-\boldsymbol{\beta}_{0}%
)\Vert^{2}-\frac{1}{\sqrt{n}}\Vert(\widehat{\boldsymbol{\beta}}_{A}-\boldsymbol{\beta}%
_{0})\Vert \Vert B_{n} \Vert+C_{n}\\
& =\frac{1}{\sqrt{n}}\Vert(\widehat{\boldsymbol{\beta}}_{A}-\boldsymbol{\beta}_{0}%
)\Vert(c_{n}\sqrt{n}\Vert(\widehat{\boldsymbol{\beta}}_{A}-\boldsymbol{\beta}%
_{0})\Vert-\Vert B_{n} \Vert \\&+ \sqrt{n} / \Vert \widehat{\bbet}_{A} -\bbet_{0}\Vert C_{n}).
\end{align*}

Now, since $Z_{n}^{2}(\widehat{\boldsymbol{\beta}}_{A})-Z_{n}^{2}%
(\boldsymbol{\beta}_{0})\leq0$, we have that.
\begin{align*}
\sqrt{n} \Vert \widehat{\boldsymbol{\beta}}_{A}-\widehat{\bbet}_{0}\Vert  &
\leq\frac{\Vert B_{n} \Vert -\sqrt{n} / \Vert \widehat{\bbet}_{A} - \bbet_{0}\Vert C_{n}}{c_{n}}.
\end{align*}

But 
\begin{equation}
\sqrt{n} / \Vert \widehat{\bbet}_{A} - \bbet_{0}\Vert C_{n} \geq -M.
\nonumber
\end{equation}

Hence, $\sqrt{n} \Vert \widehat{\bbet}_{A} - \bbet_{0}\Vert = O_{P}(1).$

\end{proof}

\begin{proof}[Proof of Theorem \ref{Spars}]

We prove (ii). The proof of (ii) is essentially the same, replacing $\iota_n$ by $\lambda_n$ and taking $\varsigma=0$.

We follow Lemma 2 of \cite{Huang}. Since by Theorem \ref{Rate} $\widehat{\bbet}_{A}$ is $\sqrt{n}$-consistent, for a sufficiently large $C>0$ and $n$, $\Vert \widehat{\bbet}_{A} - \bbet_{0} \Vert \leq C/ \sqrt{n}$ with arbitrarily high probability.

Let
\begin{align*}
V_{n}(\mathbf{u}_1,\mathbf{u}_2) &= \sum\limits_{i=1}^{n}\rho_{1}\left(  \frac{r_{i}(\bbet_{0,I}+\mathbf{u}_1/\sqrt{n},\bbet_{0,II}+\mathbf{u}_2/\sqrt{n})}{s_{n}}\right)\\
 &+\iota_{n} \left(\sum\limits_{j=1}^{s} \frac{\vert \beta_{0,j} + u_{1,j}/\sqrt{n} \vert^{t}}{\vert \widehat{\beta}_{2,j} \vert^{\varsigma}}  + \sum\limits_{j=s+1}^{p} \frac{\vert u_{2,j-s}/\sqrt{n} \vert^{t}}{\vert \widehat{\beta}_{2,j} \vert^{\varsigma}} \right).
\end{align*}

Then for large enough $n$, with arbitrarily high probability, $(\widehat{\bbet}_{A,I},\widehat{\bbet}_{A,II})$ is obtained by minimizing $V_{n}(\mathbf{u}_1,\mathbf{u}_2)$ over $\Vert \mathbf{u}_1 \Vert^{2} + \Vert \mathbf{u}_2 \Vert^{2} \leq C^{2}$. We will show that if $\Vert \mathbf{u}_1 \Vert^{2} + \Vert \mathbf{u}_2 \Vert^{2} \leq C^{2}$ and $ \Vert \mathbf{u}_2 \Vert>0$ then, for large enough $n$, $V_{n}(\mathbf{u}_1,\mathbf{u}_2)-V_{n}(\mathbf{u}_1,\mathbf{0}_{p-s})>0$ with arbitrarily high probability and the theorem will follow.

It is easy to see that
\begin{align*}
V_{n}(\mathbf{u}_1,\mathbf{u}_2)-V_{n}(\mathbf{u}_1,\mathbf{0}_{p-s}) &=
\\ \sum\limits_{i=1}^{n}\rho_{1}\left(  \frac{r_{i}(\bbet_{0,I}+\mathbf{u}_1/\sqrt{n},\mathbf{u}_2/\sqrt{n})}{s_{n}}\right)-\rho_{1}\left(  \frac{r_{i}(\bbet_{0,I}+\mathbf{u}_1/\sqrt{n},\mathbf{0}_{p-s})}{s_{n}}\right)&+\\
\frac{\iota_n}{n^{t/2}} \sum\limits_{j=s+1}^{p} \frac{\vert u_{2,j-s} \vert^{t}}{\vert \widehat{\beta}_{2,j} \vert^{\varsigma}}  &= \\
(I)+(II).
\end{align*}

Applying the Mean Value Theorem we get
\begin{equation}
(I)=(\mathbf{0}_{s},\mathbf{u}_{2})^{T}\frac{1}{\sqrt{n}}\frac{-1}{s_n}\sum\limits_{i=1}^{n}\psi_{1}\left( \frac{r_{i}(\boldsymbol{\theta}_{n}^{*})}{s_{n}}\right)\mathbf{x}_{i},
\nonumber
\end{equation}
where $\boldsymbol{\theta}_{n}^{*}=(\bbet_{0,I}+\mathbf{u}_1/\sqrt{n},(1-\alpha_n)\mathbf{u}_2/\sqrt{n})$ for some $\alpha_n \in [0,1]$. Applying the Mean Value Theorem once more we get
\begin{align*}
(\mathbf{0}_{s},\mathbf{u}_{2})^{T}\frac{1}{\sqrt{n}}\frac{-1}{s_n}\sum\limits_{i=1}^{n}\psi_{1}\left( \frac{r_{i}(\boldsymbol{\theta}_{n}^{*})}{s_{n}}\right)\mathbf{x}_{i} = \frac{1}{\sqrt{n}}\frac{-1}{s_n}(\mathbf{0}_{s},\mathbf{u}_{2})^{T}\sum\limits_{i=1}^{n}\psi_{1}\left( \frac{r_{i}(\bbet_{0})}{s_{n}}\right)\mathbf{x}_{i} &+
\\
\frac{1}{\sqrt{n}}\frac{1}{s_n^{2}}(\mathbf{0}_{s},\mathbf{u}_{2})^{T}\sum\limits_{i=1}^{n}\psi_{1}'\left( \frac{r_{i}(\boldsymbol{\theta}_{n}^{**})}{s_{n}}\right)\mathbf{x}_{i}\mathbf{x}_{i}^{T}(\mathbf{u}_1/\sqrt{n},(1-\alpha_n)\mathbf{u}_2/\sqrt{n}) &=
\\
\frac{1}{\sqrt{n}}\frac{-1}{s_n}(\mathbf{0}_{s},\mathbf{u}_{2})^{T}\sum\limits_{i=1}^{n}\psi_{1}\left( \frac{r_{i}(\bbet_{0})}{s_{n}}\right)\mathbf{x}_{i} &+ \\
 \frac{1}{n}\frac{1}{s_n^{2}}(\mathbf{0}_{s},\mathbf{u}_{2})^{T}\sum\limits_{i=1}^{n}\psi_{1}'\left( \frac{r_{i}(\boldsymbol{\theta}_{n}^{**})}{s_{n}}\right)\mathbf{x}_{i}\mathbf{x}_{i}^{T}(\mathbf{u}_1,(1-\alpha_n)\mathbf{u}_2),
\end{align*}
where $ \Vert \boldsymbol{\theta}_{n}^{**} - \bbet_{0} \Vert \leq \Vert \boldsymbol{\theta}_{n}^{*} - \bbet_{0} \Vert$.
By Lemma 5.1 of \cite{MM 85} and Lemma \ref{Lemma S unif}, the first term in the last equation is $O_{P}(\Vert \mathbf{u}_2 \Vert)$. The second term is also $O_{P}(\Vert \mathbf{u}_2 \Vert)$, by Lemma \ref{Lemma S unif}, the fact that by [B1] $\psi'_1$ is bounded, [B4] and the Law of large numbers.

On the other hand
\begin{equation}
\frac{\iota_n}{n^{t/2}} \sum\limits_{j=s+1}^{p} \frac{\vert u_{2,j-s} \vert^{t}}{\vert \widehat{\beta}_{2,j} \vert^{\varsigma}} = \iota_n n^{(\varsigma -t)/2} \sum\limits_{j=s+1}^{p} \frac{\vert u_{2,j-s} \vert^{t}}{\vert \sqrt{n} \widehat{\beta}_{2,j} \vert^{\varsigma}} = \iota_n n^{(\varsigma -t)/2} \Omega_P (\Vert \mathbf{u}_2 \Vert_{t}^{t}),
\nonumber
\end{equation}
since $\widehat{\bbet}_{2}$ is $\sqrt{n}$-consistent by assumption. Note also that $\Vert \mathbf{u}_2 \Vert_{t}^{t} \geq \Vert \mathbf{u}_2 \Vert^{t}$. Hence, for some $M_1, M_2>0$, for sufficiently large $n$ that does not depend on $\mathbf{u}_2$, with arbitrarily high probability, we have that 
\begin{align}
V_{n}(\mathbf{u}_1,\mathbf{u}_2)-V_{n}(\mathbf{u}_1,\mathbf{0}_{p-s}) > -M_1 \Vert \mathbf{u}_2 \Vert + M_2 \iota_n n^{(\varsigma -t)/2} \Vert \mathbf{u}_2 \Vert^{t} \nonumber \\ = \Vert \mathbf{u}_2 \Vert^{t} (-M_1  \Vert \mathbf{u}_2 \Vert^{1-t} + M_2 \iota_n n^{(\varsigma -t)/2}).
\label{desig-fin-sparse}
\end{align}

Finally, since by assumption $\iota_n n^{(\varsigma -t)/2} \rightarrow \infty$, we have that for any sequence of non-zero $\mathbf{u}_{2,n}$ the right hand side of \eqref{desig-fin-sparse} is strictly positive for sufficiently large $n$.

\end{proof}

\begin{proof}[Proof of Theorem \ref{Orac}]

We prove (ii). The proof of (i) is essentially the same, replacing $\iota_n$ by $\lambda_n$ and taking $\varsigma=0$.

For $\boldsymbol{\theta} \in \mathbb{R}^{s}$ let $\mathbf{p}'(\boldsymbol{\theta})=t \sum\limits_{j=1}^{s} sgn(\theta_{j}) \vert \theta_{j} \vert^{t-1} / \vert \widehat{\beta}_{2,j}\vert^{\varsigma} \mathbf{e}_{j}$. Note that by Theorem \ref{Cons}, $\widehat{\bbet}_{A}$ is strongly consistent for $\bbet_{0}$ and hence with probability 1 all the coordinates of $\widehat{\bbet}_{A,I}$ stay away from zero for a sufficiently large $n$.  Also, by Theorem \ref{Spars}, $\widehat{\bbet}_{A,II}=\mathbf{0}_{p-s}$ with probability tending to one. Then for large enough $n$, with arbitrarily high probability the partial derivatives for the first $s$ coordinates of $Z_n^{2}$ at $\widehat{\bbet}_{A}$ exist, and hence
\begin{align*}
\mathbf{0}_{s} &=\frac{1}{\sqrt{n}}\frac{-1}{s_n}\sum\limits_{i=1}^{n}\psi_{1}\left( \frac{y_{i}-\mathbf{x}_{i}^{T}\widehat{\bbet}_{A}}{s_{n}}\right)\mathbf{x}_{i,I}+\frac{\iota_n}{\sqrt{n}}\mathbf{p}'(\widehat{\bbet}_{A,I})  \\
&=\frac{1}{\sqrt{n}}\frac{-1}{s_n}\sum\limits_{i=1}^{n}\psi_{1}\left( \frac{y_{i}-\mathbf{x}_{i,I}^{T}\widehat{\bbet}_{A,I}}{s_{n}}\right)\mathbf{x}_{i,I}+\frac{\iota_n}{\sqrt{n}}\mathbf{p}'(\widehat{\bbet}_{A,I}).
\end{align*}

Then the Mean Value Theorem gives
\begin{align*}
\mathbf{0}_{s}=\frac{1}{\sqrt{n}}\frac{-1}{s_n}\sum\limits_{i=1}^{n}\psi_{1}\left( \frac{y_{i}-\mathbf{x}_{i,I}^{T}\bbet_{0,I}}{s_{n}}\right)\mathbf{x}_{i,I}&+\\
\frac{1}{s_n^{2}}\mathbf{W}_{n}\sqrt{n}(\widehat{\bbet}_{A,I}-\bbet_{0,I}) +
\frac{\iota_n}{\sqrt{n}}\mathbf{p}'(\widehat{\bbet}_{A,I}),
\end{align*}
where 
\begin{equation}
\mathbf{W}_{n}=\frac{1}{n}\sum\limits_{i=1}^{n}\psi_{1}'\left( \frac{y_{i}-\mathbf{x}_{i,I}^{T} \boldsymbol{\theta}_{n}^{*} }{s_{n}}\right)\mathbf{x}_{i,I}\mathbf{x}_{i,I}^{T}
\nonumber
\end{equation}
and $\Vert \boldsymbol{\theta}_{n}^{*} - \bbet_{0,I} \Vert \leq \Vert \widehat{\bbet}_{A,I} - \bbet_{0,I} \Vert$.

Then
\begin{align*}
\sqrt{n}(\widehat{\bbet}_{A,I}-\bbet_{0,I}) &= s_{n} \mathbf{W}_{n}^{-1} \frac{1}{\sqrt{n}}\sum\limits_{i=1}^{n}\psi_{1}\left( \frac{y_{i}-\mathbf{x}_{i,I}^{T}\bbet_{0,I}}{s_{n}}\right)\mathbf{x}_{i,I}\\ &-s_{n}^{2} \frac{\iota_n}{\sqrt{n}} \mathbf{W}_{n}^{-1}  \mathbf{p}'(\widehat{\bbet}_{A,I})
\end{align*}

By Lemma \ref{Lemma S unif}, $s_n \cas s(\bbet_0)$. By Lemma 5.1 of \cite{MM 85} and the Central Limit Theorem
\begin{equation}
\frac{1}{\sqrt{n}}\sum\limits_{i=1}^{n}\psi_{1}\left( \frac{y_{i}-\mathbf{x}_{i,I}^{T}\bbet_{0,I}}{s_{n}}\right)\mathbf{x}_{i,I} \cw \mathit{N}_{s}\left(\mathbf{0},a(\psi_1,F_0)\mathbf{V}_{\mathbf{x}_{I}}\right).
\nonumber
\end{equation}

By Lemma 4.2 of \cite{MM 85} and Lemma \ref{Lemma S unif}
\begin{equation}
\mathbf{W}_n \cas b(\psi_1,F_0) \mathbf{V}_{\mathbf{x}_{I}}.
\nonumber
\end{equation}

Since $\iota_n / \sqrt{n} \rightarrow 0$, the theorem follows from Sluztky's Theorem.

\end{proof}

\begin{proof}[Proof of Theorem \ref{Dist-Asin Knight}]
We define for $\mathbf{z} \in \mathbb{R}^{p}$
\begin{align*}
R_{n} (\mathbf{z})=  \sum\limits_{i=1}^{n}\rho_1\left(\frac{r_i(\bbet_{0}+\mathbf{z}/ \sqrt{n} )}{s_n}\right)- \rho_1\left(\frac{r_i(\bbet_{0})}{s_n}\right) \\ + \lambda_n (\sum\limits_{j=1}^{p} \vert \beta_{0,j} + z_{j}/ \sqrt{n} \vert ^q - \vert \beta_{0,j} \vert ^q),
\end{align*}
so that $\arg\min(R_{n}) = \sqrt{n}(\widehat{\boldsymbol{\beta}}_{B}-\boldsymbol{\beta}_{0})$. We will show that for each compact set $K\subset\mathbb{R}^{p}$, $R_n$ converges weakly to $R$ in $\ell^{\infty}(K)$. To do so, we will verify conditions (i) and (ii) of Theorem 2.3 of \cite{Kim}.

We first prove condition (i): finite-dimensional convergence of $R_n$ to $R$. A second order Taylor expansion shows that

\begin{align}
\label{Taylor-fini}
\sum\limits_{i=1}^{n}\rho_1\left(\frac{r_i(\bbet_{0}+\mathbf{z}/ \sqrt{n} )}{s_n}\right)-\rho_1\left(\frac{r_i(\bbet_{0})}{s_n}\right)&=\\\nonumber
-\mathbf{z}^{T}\frac{1}{s_n}\frac{1}{\sqrt{n}}\sum\limits_{i=1}^{n}\psi_1\left(\frac{r_i(\bbet_{0})}{s_n}\right)\mathbf{x}_{i}+\frac{1}{2}\frac{1}{s_n^{2}}\mathbf{z}^{T}\frac{1}{n}\sum\limits_{i=1}^{n}\psi'_1\left(\frac{u_i - \zeta_{i}\mathbf{x}_{i}^{T}  \mathbf{z}/\sqrt{n}}{s_n}\right)\mathbf{x}_{i}\mathbf{x}_{i}^{T}\mathbf{z},
\end{align}
with $0\leq \zeta_{i} \leq 1$.

It can be easily verified that for $q>1$
\begin{equation}
\lambda_n (\sum\limits_{j=1}^{p} \vert \beta_{0,j} + z_{j}/ \sqrt{n} \vert ^q - \vert \beta_{0,j} \vert ^q) \rightarrow
\lambda_{0} q \sum\limits_{i=1}^{p} z_{j} sgn(\beta_{0,j}) \vert \beta_{0,j} \vert^{q-1}
\label{conv-pen-q1}
\end{equation}
uniformly over compact sets, whereas for $q=1$
\begin{align}
\lambda_n (\sum\limits_{j=1}^{p} \vert \beta_{0,j} + z_{j}/ \sqrt{n} \vert  - \vert \beta_{0,j} \vert ) \rightarrow
\lambda_{0}  \sum\limits_{i=1}^{p} ( z_{j} sgn(\beta_{0,j}) I(\beta_{0,j}\neq0) \label{conv-pen-q2} \\+ \vert z_{j} \vert I(\beta_{0,j} = 0)),
\nonumber
\end{align}
uniformly over compact sets.

Then the finite-dimensional convergence follows from \eqref{Taylor-fini}, \eqref{conv-pen-q1}, \eqref{conv-pen-q2}, Lemma \ref{Lemma S unif}, Lemmas 4.2 and 5.1 of \cite{MM 85}, Slutzky's Theorem and the Cramer-Wold device. See the proof of Theorem \ref{Orac} for more details.

We now turn to proving condition (ii) of Theorem 2.3 of \cite{Kim}, the stochastic equicontinuity of $R_n$.
Fix $\varepsilon, \eta \text{ and } M>0$. Let $\Vert \mathbf{z} \Vert \leq M, \Vert \mathbf{z}' \Vert \leq M$.

A second order Taylor expansion shows that
\begin{align*}
\sum\limits_{i=1}^{n}\rho_1\left(\frac{r_i(\bbet_{0}+\mathbf{z}/ \sqrt{n}) }{s_n}\right)-\rho_1\left(\frac{r_i(\bbet_{0}+\mathbf{z}'/ \sqrt{n}) }{s_n}\right)&=\\\nonumber
-(\mathbf{z-z'})^{T}\frac{1}{s_n}\frac{1}{\sqrt{n}}\sum\limits_{i=1}^{n}\psi_1\left(\frac{u_i - \mathbf{x}_{i}^{T}\mathbf{z'}/\sqrt{n})}{s_n}\right)\mathbf{x}_{i}&+\\\frac{1}{2}\frac{1}{s_n^{2}}(\mathbf{z-z'})^{T}\frac{1}{n}\sum\limits_{i=1}^{n}\psi'_1\left(\frac{u_i-\zeta_i \mathbf{x}_{i}^{T} \mathbf{z}/\sqrt{n} - (1-\zeta_i)\mathbf{x}_{i}^{T}\mathbf{z'}/\sqrt{n}}{s_n}\right)\mathbf{x}_{i}\mathbf{x}_{i}^{T}(\mathbf{z-z'}),
\end{align*}
with $0\leq \zeta_i \leq 1$. Applying the Mean Value Theorem to the first term in the Taylor expansion we get
\begin{align*}
-(\mathbf{z-z'})^{T}\frac{1}{s_n}\frac{1}{\sqrt{n}}\sum\limits_{i=1}^{n}\psi_1\left(\frac{u_i - \mathbf{x}_{i}^{T}\mathbf{z'}/\sqrt{n})}{s_n}\right)\mathbf{x}_{i} = \\-(\mathbf{z-z'})^{T}\frac{1}{s_n}\frac{1}{\sqrt{n}}\sum\limits_{i=1}^{n}\psi_1\left(\frac{u_i}{s_n}\right)\mathbf{x}_{i}+\\(\mathbf{z-z'})^{T}\frac{1}{s_n^{2}}\frac{1}{n}\sum\limits_{i=1}^{n}\psi'_1\left(\frac{u_i-\kappa_i\mathbf{x}_{i}^T\mathbf{z'}/\sqrt{n}}{s_n}\right)\mathbf{x}_{i}\mathbf{x}_{i}^{T}\mathbf{z'},
\end{align*}
with $0\leq \kappa_i \leq 1$.
Then if $\Vert \mathbf{z}-\mathbf{z}' \Vert<\delta$, by Lemma \ref{Lemma S unif} and Lemmas 4.2 and 5.1 of \cite{MM 85} and the fact that by [B1] $\psi'_1$ is bounded, we have
\begin{equation}
\vert\sum\limits_{i=1}^{n}\rho_1\left(\frac{r_i(\bbet_{0}+\mathbf{z}/ \sqrt{n} )}{s_n}\right)-\rho_1\left(\frac{r_i(\bbet_{0}+\mathbf{z}'/ \sqrt{n} )}{s_n}\right)\vert \leq\delta O_{P}(1)+\delta^{2}O_{P}(1).
\label{Ineq-stoc}
\end{equation}
Let $\mathbb{P}^{*}$ stand for outer probability. Then it follows from \eqref{Ineq-stoc} that for sufficiently small $\delta$
\begin{equation}
\limsup\mathbb{P}^{*}\left(\sup \vert \sum\limits_{i=1}^{n}\rho_1\left(\frac{r_i(\bbet_{0}+\mathbf{z}/ \sqrt{n}) }{s_n}\right)-\rho_1\left(\frac{r_i(\bbet_{0}+\mathbf{z}'/ \sqrt{n}) }{s_n}\right) \vert > \eta \right)< \varepsilon,
\nonumber
\end{equation}
where the supremum runs over  $\Vert \mathbf{z} \Vert \leq M, \Vert \mathbf{z}' \Vert \leq M, \Vert \mathbf{z}-\mathbf{z}' \Vert < \delta$. Recalling \eqref{conv-pen-q1} and \eqref{conv-pen-q2} we see that we have proven condition (ii).

Since by Theorem \ref{Rate}, $\arg\min(R_{n}) = \sqrt{n}(\widehat{\boldsymbol{\beta}}_{B}-\boldsymbol{\beta}_{0})=O_{P}(1)$, the theorem follows from Theorem 2.7 of \cite{Kim}.
\end{proof}

\begin{proof}[Proof of Theorem \ref{Dist-Asin Knight2}]
After noting that
\begin{equation}
\lambda_n (\sum\limits_{j=1}^{p} \vert \beta_{0,j} + z_{j}/ \sqrt{n} \vert ^q - \vert \beta_{0,j} \vert ^q) \rightarrow \lambda_0 \sum\limits_{i=1}^{p} \vert z_j \vert^{q} I(\beta_{0,j}=0)
\nonumber
\end{equation}
uniformly over compact sets, the proof follows along the same lines as the proof of Theorem \ref{Dist-Asin Knight}.
\end{proof}

\end{document}